\documentclass[11pt,twoside]{article}
\usepackage{indentfirst}
\usepackage{bm, mathrsfs, graphics,float,amssymb,amsmath,subeqnarray,setspace,graphicx,amsthm,epstopdf,subfigure, enumerate, color,hyperref}
\def\b{\boldsymbol}

\newcommand{\tcr}[1]{\textcolor{red}{#1}}

\DeclareMathOperator{\sgn}{sgn}

\numberwithin{equation}{section}

\newtheorem{lemma}{Lemma}
\newtheorem{theorem}{Theorem}
\newtheorem{proposition}{Proposition}
\newtheorem{definition}{Definition}

\newtheorem{remark}{Remark}
\newtheorem{condition}{Condition}

\title{\Large \bf On Runge-Kutta methods for the water wave equation and its simplified nonlocal hyperbolic model}
\author{Lei Li\thanks{Department of Mathematics, Duke University; email: leili@math.duke.edu}\and Jian-Guo Liu\thanks{Department of Mathematics and Department of Physics, Duke University; email: jliu@phy.duke.edu}
\and Zibu Liu\thanks{School of Mathematical Sciences, Peking University; email: nibio21@pku.edu.cn}
\and Yi Yang\thanks{Department of Electrical Engineering, Tsinghua University, Beijing; email: yiyang16@mail.tsinghua.edu.cn} 
\and Zhennan Zhou \thanks{Beijing International Center for Mathematical Research, Peking University; email: zhennan@bicmr.pku.edu.cn}
}
\date{\today}
\begin{document}
\maketitle
\begin{abstract}
There is a growing interest in investigating numerical approximations of the water wave equation in recent years, whereas the lack of rigorous analysis of its time discretization inhibits the design of more efficient algorithms. In this work, we focus on a nonlocal hyperbolic model, which essentially inherits the features of the water wave equation, and is simplified from the latter. For the constant coefficient case, we carry out systematical stability studies of the fully discrete approximation of such systems with the Fourier spectral approximation in space and general Runge-Kutta method in time. In particular, we discover the optimal time step constraints, in the form of a modified CFL condition, when certain explicit Runge-Kutta method are applied. Besides, the convergence of the semi-discrete approximation of variable coefficient case is shown, which naturally connects to the water wave equation. Extensive numerical tests have been performed to verify the stability conditions and simulations of the simplified hyperbolic model in the high frequency regime and the water wave equation are also provided.  
\end{abstract} 
\section{Introduction}\label{sec:intro}

In the time-dependent free-surface flow problems, or the water wave problems, which date back to the early 20th century (see \cite{Milne,taylor50}), the motion of the surface has global interaction. In the past dew decades, the water wave problems has attracted much theoretic and numerical attention (see \cite{bhl96,Dull17,wu97,DiasBridges06} for examples). Since the governing equations are irrotational Euler equations, the system is not dissipating and exhibits some hyperbolic behaviors.

In \cite{beale93}, starting from the irrotational Euler equations, Beale, Hou and Lowengrub derived the equations for the fluid interface, or waterwave equations, in Lagrangian variables (see Section \ref{sec:towaterwave} for more information). By linearizing the waterwave equations, they found that the linearized equations can be changed to the following system of equations (Equation (2.8) in \cite{beale93}) for $(\alpha, t)\in \mathbb{R}\times (0,\infty)$:
\begin{gather}\label{eq:bealeerr}
\left\{
\begin{split}
\partial_t \eta&=\sigma(\alpha, t) \Lambda \zeta+g_1,\\
\partial_t \zeta&=-c(\alpha, t) \eta,\\
\partial_t\delta&=g_2.
\end{split}
\right.
\end{gather}
Here, $\alpha$ is the Lagrangian coordinate, and $\sigma$ and $c$ are positve, which depends on the solution of the waterwave equations, so independent of $\eta$ nad $\zeta$. $\eta$ is the normal component of the perturbation of the position of the interface; $\delta$ is a certain combination of the tangential and normal components of the perturbation of the position. $\zeta$ is a variable describing the perturbation of the potential. $g_1$ and $g_2$ are some extra terms in the linearization that can be controlled. The operator
\begin{gather}
\Lambda=(-\Delta)^{1/2}=H\partial_{\beta}
\end{gather}
is the $1/2$-fractional Laplacian with Fourier symbol $|k|$, where $H$ is the Hilbert transform with symbol $-i\sgn(k)$. On $\mathbb{R}$, the Hilbert transform $H$ is given by 
\[
H(f)(x)=\frac{1}{\pi}\mathrm{p.v.}\int_{-\infty}^{\infty}\frac{f(y)}{x-y}\,dy.
\]
Note that $f+iH(f)$ gives the trace of an analytic function in the upper half plane while $f-iH(f)$ gives the trace of an analytic function in the lower half plane. As we shall see, system \eqref{eq:bealeerr} is $L^2$ stable and dispersive. This system then shows the key properties of hyperbolic systems while $\Lambda$ is nonlocal. 

Note that system \eqref{eq:bealeerr} is intrinsic to the waterwave problems, no matter whether we use Lagrangian coordinate or not. Indeed, in proving the well-posedness of water wave problems in Sobolev spaces, Wu has achieved remarkable results in \cite{wu97} by using a conformal mapping formulation and reducing the water wave system to a quasi-linear hyperbolic system (see (4.6) and (5.8$\epsilon$) in \cite{wu97} and let $w=-v$) for $(\beta, t)\in \mathbb{R}\times (0,\infty)$
\begin{gather}\label{eq:wu97}
\begin{split}
& u_t=\sigma(\beta,t)\Lambda v+ b(\beta, t)\partial_{\beta}u+g_1,\\
& v_t=-c(\beta, t)u+b(\beta, t)\partial_{\beta}v+g_2
\end{split}
\end{gather}
where $\sigma>0$ and $c>0$. 
In this system $u=X_{tt}$ and $v=-X_t$ where $X$ is the $x$-coordinate of the interface. Other variables depend on the solutions. The extra derivative in $t$ plays the role of linearization in \cite{beale93}. We find that \eqref{eq:wu97} shares the same structure with \eqref{eq:bealeerr}, so the nonlocal hyperbolic system is intrinsic to waterwave problems. We say the system `hyperbolic' because it is dispersive while energy stable. Indeed, Wu used the term `hyperbolic system' in the Remark below \cite[Eq. (5.8$\epsilon$)]{wu97}. Note that there is transport terms in \eqref{eq:wu97} compared with \eqref{eq:bealeerr}. This is because $\beta$ now is not the material coordinate and it is a variable associated with the conformal mapping.

Note that if we study the periodic waves as in \cite{bhl96} or the interfaces of two dimensional drops, we then have periodic boundary conditions. This then motivates us to study the following nonlocal hyperbolic system is intrinsic to the water wave problems:
\begin{gather}\label{eq:nonlocalsystem}
\begin{split}
& u_t=\sigma(x,t)\Lambda v+g_1,\\
& v_t=-c(x, t)u+g_2
\end{split}
\end{gather}
for $(x, t)\in \mathbb{R}\times (0,\infty)$. 
If $\sigma, c$ are constant and $g_1=g_2=0$, the system is reduced to the following second order (in time) nonlocal hyperbolic equation
\begin{gather}\label{eq:simpleeq}
u_{tt}=-\mu \Lambda u,
\end{gather}
where $\mu=\sigma c$.
For heuristic purposes, we carry out some preliminary analysis and present the basic properties of \eqref{eq:simpleeq} in Section~\ref{sec:basic}.

If we consider \eqref{eq:nonlocalsystem} with periodic boundary conditions, or 
\begin{gather}\label{eq:nonlocal2}
\begin{split}
& u_t=\sigma(\theta,t)\Lambda v+g_1,\\
& v_t=-c(\theta, t)u+g_2
\end{split}
\end{gather}
with $\theta\in \mathbb{T}=\mathbb{R}/2\pi\mathbb{Z}$ and $t\in (0,\infty)$, the Hilbert transform $H$ still has symbol $-i\sgn(k)$ but the formula now is given by
\[
Hf(\theta)=\mathrm{p.v.}\int_{\mathbb{T}} f(\tau)
\cot\left(\frac{\theta-\tau}{2}\right)\frac{d\tau}{2\pi}.
\] 
$f+iH(f)$ then gives the trace of an analytic function on the unit disk.
Indeed, in studying periodic wave phenomena as in \cite{bhl96} or the motion of two dimensional drops, we have periodic boundary conditions. Studying the system on $\mathbb{T}$ makes the analysis easy while keeping the main structures. 

Numerical studies of waterwaves have been performed in many papers \cite{bhl96,dksz96,dzk96,Chalikov05,Turner16}. The numerical methods can roughly be divided into two classes depending on whether the conformal mapping is used or not. In \cite{bhl96}, the waterwave problems were solved by an integral formulation and its discretization (see Section \ref{sec:towaterwave} for more information). However, the convergence was proved with time variable being kept continuous. The discussion of the fully discretized system seems challenging. In \cite{dksz96,dzk96,Chalikov05,Turner16}, conformal mappings are used for numerical simulations but no rigorous numerical analysis for conformal mapping formulation has been performed. Meanwhile, although the analytical properties the nonlocal hyperbolic system \eqref{eq:nonlocalsystem} is relatively well understood in \cite{beale93,wu97}, the numerical studies of such equations have not been thoroughly investigated. We intend to, however, focus on numerical analysis of the simplified model \eqref{eq:nonlocalsystem}, to shed light on the distinct properties of such hyperbolic systems and waterwave simulations. Due to the presence of the nonlocal terms and the fact that the nonlocal term has a simple Fourier symbol, it is natural for one to choose the pseudo-spectral approximation in the spatial discretization, which is often favored by wave equations (see for example \cite{trefethen2000,baojinm2002}). The primary goal of this paper to analyze the Runge-Kutta methods when applied to such nonlocal wave equations. In particular, we aim to explore the optimal time step sizes, in terms of a CFL type condition, when certain explicit Runge-Kutta methods are used. As we shall show in Section~\ref{sec:basic}, the hyperbolic system \eqref{eq:nonlocalsystem} is also dispersive and may exhibit multiscale behavior, and thus, the time step constraint is more severe when the wave number is large. Consequently, finding optimal time steps with respect to the wave numbers is naturally desired. (will add a few more papers to cite here)

In the work, we have systematically analyzed stability conditions of general Runge-Kutta methods for the hyperbolic system \eqref{eq:bealeerr} with constant coefficients, including the high frequency regime, and discussed the extensions to the variable coefficient cases and to the full wave wave simulations. We have shown that, naive time discretization of the \eqref{eq:bealeerr} results in the familiar hyperbolic CFL constraint $\Delta t = O( \Delta x/K)$, which does not respect the nature of the propagation properties of \eqref{eq:bealeerr}. Here, $K$ denotes the typical wave number of the water waves. This constraint does not lead to additional challenge when $K=O(1)$. However, when $K \gg 1$, this stability condition results in {\it unnecessarily} over-resolved time steps, since it is well known that one needs to require $\Delta x=O(1/K)$ to avoid aliasing error. But, we have shown that, for typical explicit Runge-Kutta schemes whose stability regions cover part of the imaginary axis, this stability condition is improved to $\Delta t = O(\sqrt{\Delta x/K})$, and thus, the overall meshing strategy is optimal, namely, both time steps and spatial grid size only need to resolve the wave oscillation
\[
\Delta t=O(1/K), \quad \Delta x= O(1/K).
\] 
This result is sharp in the view that one cannot capture the accurate wave function without resolving its oscillations, and hence, it greatly facilitates efficient simulations of the water wave problem.

The rest of the paper is organized as follows: in Section \ref{sec:setup}, we introduce the basic notations and setup for the numerical analysis; in Section \ref{sec:constantcoe}, we discuss thoroughly the discretizations of the nonlocal system with Runga-Kutta (both explicit and implicit) method in time and Fourier spectral method in space. We then study the discretization of the system with variable coefficients in Section \ref{sec:variable}. We prove the convergence for the semi-discrete schemes using Fourier spectral method or filtered Fourier spectral method, and then discuss the time discretizations using Runga-Kutta methods. We then connect the nonlocal hyperbolic system to waterwave equations in Section \ref{sec:towaterwave}. Lastly, in Section \ref{sec:num}, we perform numerical experiments. The stability conditions for the nonlocal hyperbolic system with variable coefficients and waterwave equations are confirmed numerically. Numerical experiments suggest possible caustics for the system in the high frequency regimes.

\subsection{Basic properties of the nonlocal hyperbolic equation} \label{sec:basic}

In this section, we present a concise review of basic properties of \eqref{eq:simpleeq}, a special case of the hyperbolic system \eqref{eq:bealeerr}. Thus, it suffices to consider the simplified version, the second order wave equation.

Multiplying by $u_t$ on both sides of \eqref{eq:simpleeq}, and integrating over $x$ yields
\[
\frac{d}{dt}\int_{\mathbb{R}} \left(|u_t|^2+ \frac{1}{2}\mu u\Lambda u \right) \,dx=0,
\]
which means the energy
\begin{gather}\label{eq:energy}
E=\int_{\mathbb{R}}\left(|u_t|^2+ \frac{1}{2}\mu u\Lambda u \right) dx,
\end{gather}
is a conserved quantity in time. The dispersion relation can be derived in the following way. On the Fourier side, \eqref{eq:simpleeq} can be written as
\begin{gather}
\hat{u}_{tt}=-\mu|\xi|\hat{u},
\end{gather}
and the Fouriere Transform of plane wave $u(x,t)=Ae^{i(kx-wt)}$ is $\hat{u}(\xi,t)=A\delta(k-2\pi\xi)e^{-iwt}$. Let $\hat{u}(\xi,t)$ be a solution of it, we can get
\begin{gather}\label{eq:dispersive}
-\omega^2=-\mu |\xi|\Rightarrow \omega=\pm\sqrt{\mu |\xi|},
\end{gather}which is the dispersion relation. Because $\omega$ is relative to $\xi$, so the system is dispersive. Remember that we have proved the system is also energy stable by formula \eqref{eq:energy}, so the system is 'hyperbolic'. But due to the fractional Laplacian operator, the equation is non-local in space.
To have a better understanding of the system, we move on to the fundamental solution of the following Cauchy problem: 
\begin{gather}\label{eq:Green}
\left\{
\begin{split}
& u_{tt}+\mu\Lambda u=0,(x,t)\in\mathbb{R}\times\mathbb{R}^+\\
& u(x,0)=\delta(x),\\
& u_t(x,0)=0.
\end{split}
\right.
\end{gather}
On the Fourier side:
\begin{gather}\label{eq:GreenFourier}
\left\{
\begin{split}
& \hat{u}_{tt}+\mu|\xi|\hat{u}=0,(\xi,t)\in\mathbb{R}\times\mathbb{R}^+\\
& \hat{u}(\xi,0)=1,\\
& \hat{u}_t(\xi,0)=0.
\end{split}
\right.
\end{gather}
\eqref{eq:GreenFourier} is a second order ODE initial value problem and the unique solution is
\begin{gather}
\hat{u}(\xi,t)=\mathrm{cos}(\sqrt{\mu|\xi|}t).
\end{gather}
Performing inverse Fourier transform, the fundamental solution of \eqref{eq:Green} is
\begin{gather}
G(x,t)=\int_{\mathbb{R}}\mathrm{cos}(\sqrt{\mu|\xi|}t)e^{2\pi i\xi x}\mathrm{d}\xi.
\end{gather}
Thus, in the following general case,
\begin{gather}\label{eq:general}
\left\{
\begin{split}
& u_{tt}+\mu\Lambda u=0,(x,t)\in\mathbb{R}\times\mathbb{R}^+\\
& u(x,0)=f(x),\\
& u_t(x,0)=0,
\end{split}
\right.
\end{gather}
the solution is the convolution of $f(x)$ and $G(x,t)$. An easy calculation shows that the solution has a equivalent express:
\begin{gather}
u(x,t)=\int_{\mathbb{R}}\hat{f}(\xi)\mathrm{cos}(\sqrt{\mu|\xi|}t)e^{2\pi i\xi x}\mathrm{d}\xi,
\end{gather}
where $\hat{f}(\xi)$ is the Fourier Transform of $f(x)$. Let $f(x)=\mathrm{cos}(kx)$, then the solution is cos($kx$)cos($\sqrt{\mu|k|}t$), this is also reasonable for we have the dispersion relation \eqref{eq:dispersive}.
\begin{remark}
Equation \eqref{eq:simpleeq} is reminiscent of the surface quasi-geostrophic equations
(SQG) studied in \cite{knv07,cv10}. However, the surface SQG equation is dissipating while \eqref{eq:simpleeq} is dispersive.
\end{remark}

\section{Notations and setup}\label{sec:setup}

In this work, we consider the Fourier spectral method or the filtered Fourier spectral method for the spatial discretization of the one dimensional nonlocal hyperbolic system \eqref{eq:nonlocal2} on $\mathbb{T}=\mathbb{R}/2\pi\mathbb{Z}$.  

We discretize the spatial domain with grid size $h=2\pi/N$, and we denote grid points by $\theta_j=jh$, $j \in [N]=\{1,\cdots,N\}$, where $N\in \mathbb{N}$ is even. We denote the time step size by $\tau$, and denote $t^n=n\tau$. The notation $u_j^n$ represents the numerical value of $u(\theta,t)$ at $(\theta_j, t^n)$, and $u^n$ represents the vector $\b{u}^n=(u_j^n)$.  

Given any $N$-vector $\b{f}=(f_j)$, we expand each component as a sum of discrete Fourier modes via
\begin{gather*}
f_j=\sum_{k\in[N]^*}\hat{f}_{k} e^{i k\theta_j}, ~~~j\in[N],
\end{gather*}
where $[N]^*:=\{-\frac{1}{2}N+1,\ldots,\frac{1}{2}N\}$ and the discrete Fourier transform $\hat{f}=(\hat{f}_k)$ are given by
\begin{gather*}
\hat{f}_k=\frac{1}{N}\sum_{j\in[N]}f_je^{-ik\theta_j}, ~~~k\in[N]^*.
\end{gather*}

Note that the Hilbert transform $H$ and the derivative of a function becomes certain multipliers when the Fourier transform is applied. When projected onto a uniform grid, those transforms between two function reduce to corresponding relations between the discrete Fourier transforms of two functions confined on the grid.

We define the projected differential operator and the projector Hilbert transform $H$ in the following. 
For two $N-$ vector $\b{f}$ and $\b{g}$, we write
\begin{gather*}
\b{g}=\mathcal{D}\b{f} ~~~\text{to mean}~~~ \hat{g}_k=ik\hat{f}_k, ~~k\in[N]^*, \\
\b{g}=\mathcal{H}\b{f} ~~~\text{to mean}~~~ \hat{g}_k=-i{\rm sgn}(k)\hat{f}_k, ~~k\in[N]^*.
\end{gather*}
We introduce the notation $\mathcal{L}=\mathcal{D}\mathcal{H}$ as the projected $\Lambda=\partial H$, so that
\begin{gather*}
\b{g}=\mathcal{L}\b{f} ~~~\text{means}~~~ \hat{g}_k=\left|k\right|\hat{f}_k, ~~ k\in[N]^*.
\end{gather*}

We use $\mathscr{E}_N$ to represent the set of $N$-vectors. Recall the discrete inner product between two $N$-vectors is defined as
\begin{gather*}
\langle \b{f},\b{g} \rangle=\sum_{j\in[N]}hf_j\bar{g}_j,
\end{gather*}
where $\bar{g}$ means the complex conjugate.
The discrete $\ell^2$ and $\ell^{\infty}$ norms are defined by
\begin{gather*}
\Vert f\Vert_2=\sqrt{\langle \b{f}, \b{f} \rangle}, ~~~\Vert \b{f} \Vert_{\infty}=\max_{j\in[N]}|f_j|.
\end{gather*}
\begin{lemma}\label{lmm:pars}
The discrete Parserval's equality holds
\begin{gather*}
\langle \b{f},\b{g} \rangle=\sum_{j\in[N]}hf_j\bar{g}_j=2\pi\sum_{k\in[N]^*}\hat{f}_k\bar{g_k},
\end{gather*}
\end{lemma}

\section{Discretization of the constant-coefficient equations}\label{sec:constantcoe}

We explore the optimal time discretization method in terms of CFL-type stability constraints in this section for the constant-coefficient nonlocal hyperbolic equation \eqref{eq:nonlocal2} for $(\theta, t)\in \mathbb{T}\times (0,\infty)$:
\begin{gather}\label{eq:constantcoe}
\left\{
\begin{split}
& u_t=\sigma \Lambda v,\\
& v_t=-c\, u.
\end{split}
\right.
\end{gather}
The system can be written on the Fourier side as
\begin{equation} \label{sys:Fourier}
\partial_{t}\left(\begin{array}{c} \hat{u} \\
\hat{v} \end{array}\right) 
=
\left(  
\begin{array}{ccc}  
0 & \sigma |k| \\  
-c & 0 \\
\end{array}
\right)
\left(\begin{array}{c} \hat{u} \\
\hat{v} \end{array}\right) :=A\left(\begin{array}{c} \hat{u} \\
\hat{v} \end{array}\right). 
\end{equation}

We derive in the following the stability conditions of general Runge-Kutta methods when applied to this system \eqref{sys:Fourier}, including explicit and implicit schemes. In particular, we aim to investigate optimal stability conditions when explicit Runge-Kutta methods are used. We define notations for the Butcher tableau of a certain $n$-step RK method. The Butcher tableau is given by:
\\
\[
\renewcommand\arraystretch{1.2}
\begin{array}
{c|c}
\b{p} & \b{G} \\
\hline
&\b{w}^T
\end{array}
,\]
where $\b{G}$ is the Runge-Kutta matrix, $\b{w}$ are the weights and $\b{p}$ are the nodes.\par
\subsection{Preliminary Analysis on $\b{A}$, an equivalent system and its Growth Matrix}
In this section, we will first do some preliminary analysis on $\b{A}$ as in \eqref{sys:Fourier}, derive an equivalent system to facilitate analysis, and carry out Von Neumann analysis of the derived system. We show the Von Neumann analysis is greatly simplified by analyzing the spectral radius of the growth matrix correspond to the derived system when a certain RK method is used to solve it.\par  
An easy calculation shows that matrix $\b{A}$ has 2 complex eigenvalues: $\lambda_{1,2}=\pm i\sqrt{c\sigma |k|}$. Notice that the 2 eigenvalues are both pure imaginary, thus $A$ is similar to an antisymmetric matrix $\b{Q}$ in $\mathbb{R}$:
\begin{equation} \label{Asimilar}
\b{A}=\b{P}^{-1}\b{QP},\, \b{Q}=\left(  
\begin{array}{ccc}  
0 & \sqrt{c\sigma|k|} \\  
-\sqrt{c\sigma|k|} & 0 \\
\end{array}
\right),\, \b{P}=\left(  
\begin{array}{ccc}  
1 & 0 \\  
0 & \sqrt{\dfrac{\sigma|k|}{c}} \\
\end{array}
\right).
\end{equation}
Substitute \eqref{Asimilar} in \eqref{sys:Fourier}, we will get a new system which has equivalent stable condition as \eqref{sys:Fourier}：
\begin{equation} \label{sys:QFourier}
\partial_{t}\left(\begin{array}{c} \hat{u} \\
\hat{w} \end{array}\right) 
=
\b{Q}
\left(\begin{array}{c} \hat{u} \\
\hat{w} \end{array}\right),\, \left(\begin{array}{c} \hat{u} \\
\hat{w} \end{array}\right):=\b{P}\left(\begin{array}{c} \hat{u} \\
\hat{v} \end{array}\right).
\end{equation}
Therefore, focusing on system \eqref{sys:QFourier} is sufficient.\par

\begin{remark}
Note that
\[
\|w\|_2^2=\frac{\sigma}{c}\sum_{k\in [N]^*}|k||\hat{v}_k|^2
=\frac{\sigma}{c}\langle\mathcal{L}v, v\rangle.
\]
By observing energy \eqref{eq:energy}, the $H^{1/2}$ norm of $v$ is essential to the system while introducing matrix $\b{P}$ indeed implies that we use the $H^{1/2}$ norm for $v$.
\end{remark}

Besides being antisymmetric, $\b{Q}$ also has a special property which is crutial in the following analysis:
\begin{proposition}$\b{Q}^2=-c\sigma|k|\b{I}.$
\end{proposition}
The proof follows by direct calculation.\par
To accomplish Von Neumann Analysis, we need to analyze the norm of the growth matrix (let it be $\b{M}$) when a certain RK method is used to solve \eqref{sys:QFourier}. Because $M$ is a polynomial (when explicit RK method is applied) or a rational function (when implicit RK method is applied) of $\tau \b{Q}$, we also need the connection between $\b{Q}$ and $\b{M}$. Using proposition 1, we can show that $\b{M}=a(\tau^2|k|)\b{I}+b(\tau^2|k|)\tau \b{Q}$, where $a(x),b(x)$ are constant coefficient rational functions, which is Lemma 1:
\begin{lemma}
For system \eqref{sys:QFourier}, any RK method applied to it reduces to a iteration:
\begin{equation}
\left(\begin{array}{c} \hat{u}^{n+1} \\
\hat{w}^{n+1} \end{array}\right)=\b{M}\left(\begin{array}{c} \hat{u}^n \\
\hat{w}^n \end{array}\right),
\end{equation}
where $\b{M}=a(\tau^2|k|)\b{I}+b(\tau^2|k|)\tau \b{Q}$, $\b{Q}$ is the one in \eqref{Asimilar} and $a(x),b(x)$ are constant coefficient rational functions.
\end{lemma}
\begin{proof}
$\b{M}$ can be written in the following way:
\begin{equation}\label{M}
M=f(\tau \b{Q})[g(\tau \b{Q})]^{-1},
\end{equation}where $f(x),g(x)$ are constant coefficient polynomials. Due to proposition 1, we get
\begin{equation}
(\tau \b{Q})^2=-\tau^2|k|c\sigma \b{I}.
\end{equation}
Thus we can rewrite $f(\tau \b{Q}),g(\tau \b{Q})$ as:
\begin{equation}\label{fg}
f(\tau \b{Q})=a_1(\tau^2|k|)\b{I}+b_1(\tau^2|k|)\tau \b{Q},g(\tau \b{Q})=a_2(\tau^2|k|)\b{I}+b_2(\tau^2|k|)\tau \b{Q},
\end{equation}where $a_1(x),a_2(x),b_1(x),b_2(x)$ are also constant coefficient polynomials.
Still using proposition 1, we can get the explicit formula of $[g(\tau \b{Q})]^{-1}$. Observe that  
\begin{equation}
g(\tau \b{Q})(a_2(\tau^2|k|)I\b{I}-b_2(\tau^2|k|)\tau \b{Q})=[a_2^2(\tau^2|k|)+c\sigma\tau^2|k|b_2^2(\tau^2|k|)]\b{I}:=h(\tau^2|k|)\b{I}.
\end{equation}Thus, 
\begin{equation}\label{invg}
[g(\tau \b{Q})]^{-1}=\dfrac{a_2(\tau^2|k|)\b{I}-b_2(\tau^2|k|)\tau \b{Q}}{h(\tau^2|k|)}.
\end{equation}
Substitute \eqref{invg},\eqref{fg} in \eqref{M}, we get:
\begin{equation}
\begin{split}
M&=f(\tau \b{Q})[g(\tau \b{Q})]^{-1}\\
&=\left[a_1(\tau^2|k|)\b{I}+b_1(\tau^2|k|)\tau \b{Q}\right]\dfrac{a_2(\tau^2|k|)\b{I}-b_2(\tau^2|k|)\tau \b{Q}}{h(\tau^2|k|)}\\
&:=a(\tau^2|k|)\b{I}+b(\tau^2|k|)\tau \b{Q},
\end{split}
\end{equation}where
\begin{equation}
a(x)=\dfrac{a_1(x)a_2(x)-c\sigma xb_1(x)b_2(x)}{h(x)},
b(x)=\dfrac{b_1(x)a_2(x)-a_1(x)b_2(x)}{h(x)}.
\end{equation}
Because $a_1(x),a_2(x),b_1(x),b_2(x),h(x)$ are all polynomials, $a(x),b(x)$ are both rational functions.
\end{proof}
Now that $M$ has a specific form of $\b{M}=a(\tau^2|k|)\b{I}+b(\tau^2|k|)\tau \b{Q}$, we observed that there is a relationship between norm of $\b{M}$ and $\rho(\b{M})$, which is the following lemma:
\begin{lemma}
If matrix $\b{M}=p(\tau,|k|)\b{I}+q(\tau,|k|) \b{Q}$, then\begin{equation}
\rho(\b{M})=\sqrt{\rho(\b{M}^T\b{M})}=\|\b{M}\|_2,
\end{equation}where $\b{Q}$ is the one in \eqref{Asimilar} and $p(\tau,|k|),q(\tau,|k|)$ are any real functions.
\end{lemma}
\begin{proof}
Because $\b{Q}^T=-\b{Q}$, an easy calculation tells us:\begin{equation}
\begin{split}
\b{M}^T\b{M}&=\left[p(\tau,|k|)\b{I}+q(\tau,|k|) \b{Q}\right]\left[p(\tau,|k|)\b{I}-q(\tau,|k|) \b{Q}\right]\\
&=p^2(\tau,|k|)\b{I}-q^2(\tau,|k|)\b{Q}^2\\
&=\left[p^2(\tau,|k|)+c\sigma|k|q^2(\tau,|k|)\right]\b{I}.
\end{split}
\end{equation}Thus, 
\begin{equation}
\sqrt{\rho(\b{M}^T\b{M})}=\sqrt{p^2(\tau,|k|)+c\sigma|k|q^2(\tau,|k|)}.
\end{equation}
For the 2 eigenvalues of matrix $\b{Q}$ are $\lambda_{1,2}=\pm i\sqrt{c\sigma|k|}$, thus
\begin{equation}
\rho(\b{M})=\left|p(\tau,|k|)+iq(\tau,|k|)\sqrt{c\sigma|k|}\right|=\sqrt{p^2(\tau,|k|)+c\sigma|k|q^2(\tau,|k|)}.
\end{equation}Hence, we have shown that
\begin{gather*}
\rho(\b{M})=\sqrt{\rho(\b{M}^T\b{M})}=\|\b{M}\|_2.
\end{gather*}
\end{proof}
This observation is meaningful for it warrants that we can estabilsh the framework by merely analizing $\rho(M)$. The advantage of analysis $\rho(M)$ instead of $\|M\|_2$ directly lies in that although they are exactly the same in this case, there are already general results say that the spectral radius of growth matrix can be explicitly represented by $G,w,p$. Due to this reason, we will then establish our general framework by analysis on eigenvalues of $M$.

\subsection{Stability Analysis}
Now the rest of work only lies in the analysis of eigenvalues of matrix $\b{M}$. Let $\nu=c\sigma|k|$, then two eigenvalues of matrix $\b{Q}$ are $\lambda_{1,2}=\pm i\sqrt{c\sigma|k|}=\pm i\sqrt{\nu}$, assume that matrix $\b{P}_1$ satisfies 
\begin{equation} \label{Qsimilar}
\b{Q}=\b{P}_1^{-1}\b{\Lambda}_1 \b{P}_1,\, \b{\Lambda}_1=\mathrm{diag}\{\lambda_1,\lambda_2\}.
\end{equation}
Substitute \eqref{Qsimilar} in \eqref{sys:QFourier}, we can get a decoupled system:
\begin{equation} \label{sys:Qsimilar}
\partial_{t}\left(\begin{array}{c} \hat{u}_1 \\
\hat{v}_1 \end{array}\right) 
=\b{\Lambda}_1\left(\begin{array}{c} \hat{u}_1 \\
\hat{v}_1 \end{array}\right),\, 
\left(\begin{array}{c} \hat{u}_1\\ \hat{v}_1 \end{array}\right):=
\b{P_1}\left(\begin{array}{c} \hat{u}\\ \hat{w} \end{array}\right). 
\end{equation}  
Given that $\b{\Lambda}_1$ is similar to $\b{Q}$, so when using the same RK method to solve \eqref{sys:QFourier} and \eqref{sys:Qsimilar} respectively, the spectral radius of growth matrices are exactly the same. Remember that section 3.1 has already shown that anaylsis on $\rho(\b{M})$ is sufficient, because $\rho(\b{M})=\|\b{M}\|_2$, thus it is also sufficient to focus on \eqref{sys:Qsimilar}. \par
Now that $\b{\Lambda}_1$ is a diagonal complex matrix, thus function $\hat{u}_1,\hat{v}_1$ are decoupled, and conclusion about the eigenvalues of growth matrix for any RK method, namely Lemma 3, can be directly used to derive the explicit formula of eigenvalues:
\begin{lemma}
Let $\b{e}$ stands for vector of ones. For the linear test equation $y'=\lambda y$, the RK method applied to this equation reduces to $y_{n+1}=f(\tau\lambda)y_n$, with $f(z)$ given by
\begin{equation} \label{f(z)}
f(z)=1+z\b{w}^T(\b{I}-z\b{G})^{-1}\b{e}=\dfrac{\mathrm{det}\left(\b{I}-z\b{G}+z\b{e}\b{w}^T\right)}{\mathrm{det}\left(\b{I}-z\b{G}\right)}.
\end{equation}
\end{lemma}
This Lemma can be found in \cite{hairer1996solving}. By Lemma 3, the 2 eigenvalues of the growth matrix are $f(\tau\lambda_1),f(\tau\lambda_2)$, namely
\begin{equation}\label{mu1mu2}
\mu_1=f(\tau\lambda_1)=\dfrac{\mathrm{det}\left(\b{I}-\tau\lambda_1\b{G}+\tau\lambda_1\b{e}\b{w}^T\right)}{\mathrm{det}\left(\b{I}-\tau\lambda_1\b{G}\right)},\mu_2=f(\tau\lambda_2)=\dfrac{\mathrm{det}\left(\b{I}-\tau\lambda_2\b{G}+\tau\lambda_2\b{e}\b{w}^T\right)}{\mathrm{det}\left(\b{I}-\tau\lambda_2\b{G}\right)}.
\end{equation}\par
For we already have formula \eqref{mu1mu2}, to derive a formula for $|\mu_1|,|\mu_2|$, we need the following lemma: 
\begin{lemma}
Complex matrix $\b{C}=\b{A}+i\b{B}$ where $\b{A}$ and $\b{B}$ are real matrices satisfy $\b{AB}=\b{BA}$, then $\left|\mathrm{det}\left(\b{C}\right)\right|=\sqrt{\left|\mathrm{det}(\b{A}^2+\b{B}^2)\right|}$.
\end{lemma}
\begin{proof}
$\left|\mathrm{det}\left(\b{C}\right)\right|=\left|\mathrm{det}\left(\b{\overline{C}}\right)\right|=\sqrt{\left|\mathrm{det}\left(\b{C}\overline{\b{C}}\right)\right|}$. Meanwhile
$\b{C}\overline{\b{C}}=(\b{A}+i\b{B})(\b{A}-i\b{B})=\b{A}^2+\b{B}^2$, thus $\left|\mathrm{det}\left(\b{C}\right)\right|=\sqrt{\left|\mathrm{det}(\b{A}^2+\b{B}^2)\right|}$.
\end{proof} 
Using lemma 4, let $\b{A}=\b{I}$,\, $\b{B}=\tau\sqrt{\nu}\b{G},\tau\sqrt{\nu}(\b{G}-\b{e}\b{w}^T)$, we can get
\begin{equation} \label{eigmodulus}
|\mu_1|=|\mu_2|=\sqrt{\dfrac{\left|\mathrm{det}(\b{I}+\tau^2\nu(\b{G}-\b{e}\b{w}^T)^2)\right|}{\left|\mathrm{det}(\b{I}+\tau^2\nu \b{G}^2)\right|}}.
\end{equation} Denote $\psi(\tau,\nu)$ as $\sqrt{\dfrac{\left|\mathrm{det}(\b{I}+\tau^2\nu(\b{G}-\b{e}\b{w}^T)^2)\right|}{\left|\mathrm{det}(\b{I}+\tau^2\nu \b{G}^2)\right|}}$.
For explicit RK methods, matrix $\b{G}$ is lower triangular matrix, thus $\mathrm{det}(\b{I}+\tau^2\nu\b{G}^2)=1$, and formula \eqref{eigmodulus} reduces to
\begin{equation}
|\mu_1|=|\mu_2|=\sqrt{\left|\mathrm{det}(\b{I}+\tau^2\nu(\b{G}-\b{e}\b{w}^T)^2)\right|},
\end{equation} which means that $|\mu_1|^2$, or $|\mu_2|^2$ is a constant polynomial coefficient of $\tau^2\nu$, and the polynomial is relevant to $\b{G},\b{w}$. This indeed makes sense, because when using explicit RK-p method to solve system \eqref{sys:Qsimilar}, the following formula still holds:
\begin{equation}
\mu_{1,2}=\sum_{m=0}^p\dfrac{1}{m!}(\tau \lambda_{1,2})^m=\sum_{m=0}^p\dfrac{1}{m!}(\pm i\tau\sqrt{\nu})^m,
\end{equation}which is a well known property of explicit RK method. Therefore, $|\mu_1|^2$, or $|\mu_2|^2$ is a constant polynomial coefficient of $\tau^2\nu$. Besides, when $p\geqslant 3$, it is also well known that the absolute stable region of RK-p method contains a part of imaginary axis $[-iC_1(p),iC_1(p)]$, thus asking
\begin{equation}\label{CFL_RKp}
-C_1(p)\leqslant \pm \tau\sqrt{\nu}\leqslant C_1(p)
\end{equation} is sufficient to ensure the stability of explicit RK-p method when solving system \eqref{sys:Qsimilar}. Combining this observation and formula \eqref{eigmodulus}, we can prove the following theorem which characterizes the stability condition for any RK method:  
\begin{theorem} There are 2 cases in this theorem:
\begin{enumerate}
\item For any RK method, if a CFL condition holds, i.e. $\tau\leqslant Ch$ for some positive real number $C$, the scheme for system \eqref{sys:Qsimilar}, \eqref{sys:Fourier} or \eqref{sys:QFourier} is stable. More specifically, let $\b{e}$ stand for vector of ones, then 
\begin{enumerate}
\item If $\mathrm{tr}(\b{G}^2)>\mathrm{tr}\left((\b{G}-\b{e}\b{w}^T)^2\right)$, then $\psi(\tau,\nu)<1$ when $\tau$ goes to 0, which means strong stability.
\item If $\mathrm{tr}(\b{G}^2)\leqslant\mathrm{tr}\left((\b{G}-\b{e}\b{w}^T)^2\right)$, then 
$\psi(\tau,\nu)\leqslant1+L\tau+O(\tau^2)$ when $\tau$ goes to 0, where $0\leqslant L\leqslant\mathrm{tr}\left((\b{G}-\b{e}\b{w}^T)^2-\b{G}^2\right)c\sigma C\pi/2$, which means weak stability.
\end{enumerate} 
\item For any RK method whose absolute stable region contains a part of imaginary axis $[-C_1i,C_1i]$, then there exists $C_2>0$ such that when
\begin{equation}\label{eq:newcfl}
\tau\leqslant C_2\sqrt{h},
\end{equation} 
the scheme is strongly stable.
\end{enumerate}
\end{theorem}
\begin{proof}
We first aim to prove statement (a) and (b). Consider $\mathrm{det}\left(\b{I}+\tau \b{X}\right)$, in which $\tau$ is a small number and $\b{X}$ is a constant matrix. Suppose that the $n$ eigenvalues (counting multipility) are $\lambda_1,\lambda_2,...,\lambda_n$, then 
the determinant of $\b{X}$ is $\prod_{i=1}^n\lambda_i$. And,
\begin{equation}\label{determinant}
\mathrm{det}\left(\b{I}+\tau \b{X}\right)=\prod_{i=1}^n\left(1+\tau\lambda_i\right).
\end{equation}
Expand the R.H.S. of \eqref{determinant}, we can get 
\begin{equation}\label{smallnumber}
\begin{split}
\mathrm{det}\left(\b{I}+\tau \b{X}\right)&=1+\sum_{i=1}^n\lambda_i\tau+O(\tau^2)\\
&=1+\mathrm{tr}\left(\b{X}\right)\tau+O(\tau^2).
\end{split}
\end{equation}
Recall that $\nu=c\sigma|k|$ and $|k|h\leqslant\pi$, we can get
\begin{equation}
\begin{split}
\tau\nu=c\sigma|k|\tau&\leqslant c\sigma\pi\tau/h\\
&\leqslant c\sigma C\pi,
\end{split}
\end{equation} 
thus $\tau\nu=O(1)$ and $\tau^2\nu=O(\tau)$, we can directly use \eqref{smallnumber} in $\psi(\tau,\nu)$, i.e.
\begin{equation}
\begin{split}
\psi(\tau,v)&=\sqrt{\dfrac{\left|\mathrm{det}(\b{I}+\tau^2\nu(\b{G}-\b{e}\b{w}^T)^2)\right|}{\left|\mathrm{det}(\b{I}+\tau^2\nu \b{G}^2)\right|}}\\
&=\sqrt{\dfrac{\left|1+\mathrm{tr}\left((\b{G}-\b{e}\b{w}^T)^2\right)\tau^2\nu+O(\tau^2)\right|}{\left|1+\mathrm{tr}\left(\b{G}^2\right)\tau^2\nu+O(\tau^2)\right|}}\\
&=\sqrt{1+\left[\mathrm{tr}\left((\b{G}-\b{e}\b{w}^T)^2\right)-\mathrm{tr}\left(\b{G}^2\right)\right]\tau^2\nu+O(\tau^2)}\\
&=1+\dfrac{\mathrm{tr}\left((\b{G}-\b{e}\b{w}^T)^2\right)-\mathrm{tr}\left(\b{G}^2\right)}{2}\tau^2\nu+O(\tau^2) 
\end{split}.
\end{equation}
Then we consider 2 cases respectively:
\begin{enumerate}
\item If $\mathrm{tr}(\b{G}^2)>\mathrm{tr}\left((\b{G}-\b{e}\b{w}^T)^2\right)$, then $\psi(\tau,\nu)<1$ when $\tau$ goes to 0, which means strong stability.
\item If $\mathrm{tr}(\b{G}^2)\leqslant\mathrm{tr}\left((\b{G}-\b{e}\b{w}^T)^2\right)$, subsitute (1.9) in the last equality of (1.10), we can show that 
$\psi(\tau,\nu)\leqslant1+L\tau+O(\tau^2)$ when $\tau$ goes to 0, where
\begin{equation}
L=\dfrac{\mathrm{tr}\left((\b{G}-\b{e}\b{w}^T)^2-\b{G}^2\right)\nu C\pi}{2},
\end{equation}which means weak stability.
\end{enumerate}
For the third statement, we first consider the range of $\tau\sqrt{\nu}$:
\begin{equation}
\begin{split}
\tau\sqrt{\nu}&=\tau\sqrt{c\sigma|k|}\\
&=\tau\sqrt{c\sigma\pi/h}.
\end{split}
\end{equation}Thus, if $\tau\sqrt{c\sigma\pi/h}\leqslant C_1$, namely
\begin{equation}
\tau\leqslant\dfrac{C_1}{\sqrt{c\sigma\pi}}\sqrt{h}:=C_2\sqrt{h},
\end{equation}the scheme is stable.
\end{proof}
To check the validity of Theorem 1, we give two examples of RK method, who are semi-implicit RK method and explicit RK method respectively. We can find that the condition we need to ensure stability may be weaker than $\tau\leqslant Ch$ when using a certain RK method, even the stability may be unconditionally strong stability. In addition, a conditon claimed in the second statement of Theorem 1, namely $\tau\leqslant C\sqrt{h}$, can also be observed.
\begin{enumerate}[(a)]
\item Semi-implicit RK2: weighted Euler. The tableau of weighted Euler is:\\
\[
\renewcommand\arraystretch{1.2}
\begin{array}
{c|cc}
0 &0 &0 \\
1& 1-\delta& \delta\\
\hline
& 1-\delta &\delta
\end{array}
\]
where we can see that
\begin{equation*}
\b{G}=
\begin{pmatrix}
0 & 0\\
1-\delta & \delta
\end{pmatrix},\b{w}^T=
(1-\delta,\delta).
\end{equation*}
Substitute them in $\psi(\tau,\nu)$, we can get 
\begin{equation}
\psi(\tau,\nu)=\sqrt{\dfrac{1+(1-\delta)^2\tau^2\nu}{1+\delta^2\tau^2\nu}}.
\end{equation} 
Although we have proved that $\tau\leqslant Ch$ can ensure stability, we can also find that if $\delta\geqslant\dfrac{1}{2}$, then $\psi(\tau,\nu)\leqslant 1$, which means unconditionally strong stability. This means that condition of $\tau\leqslant Ch$ is not a sharp one.
\item Explicit RK4. The tableau of explicit RK4 is:\\
\[
\renewcommand\arraystretch{1.2}
\begin{array}
{c|cccc}
0\\
\frac{1}{2} & \frac{1}{2}\\
\frac{1}{2} &0 &\frac{1}{2} \\
1& 0& 0& 1\\
\hline
& \frac{1}{6} &\frac{1}{3} &\frac{1}{3} &\frac{1}{6} 
\end{array}
\]
where we can see that
\begin{equation*}
\b{G}=
\begin{pmatrix}
0 & 0 & 0 & 0\\
\dfrac{1}{2} & 0 & 0 &0\\
0 & \dfrac{1}{2} & 0 &0\\
0 & 0 & 1 & 0\\
\end{pmatrix},\b{w}^T=\left(\dfrac{1}{6},\dfrac{1}{3},\dfrac{1}{3},\dfrac{1}{6}\right).
\end{equation*}
Notice that the absolute stable region of explicit RK-4 contains a part of imaginary axis, by Theorem 1, we should expect that there exists a CFL condition like $\tau\leqslant C\sqrt{h}$. Because this RK method is explicit, so $\mathrm{det}(\b{I}+\tau^2\nu \b{G}^2)=1$, and we only need to compute the numerator of $\psi(\tau,\nu)$.
Substitute $\b{G},\b{w}^T$ in $\psi(\tau,\nu)$, and let $z=\tau^2\nu$, then we can get 
\begin{equation}
\rho(\tau,\nu)=\sqrt{\mathrm{det}(1+z(G-eb^T)^2)}=\sqrt{\dfrac{z^4}{576}-\dfrac{z^3}{72}+1}.
\end{equation}
If we require $\psi(\tau,\nu)\leqslant1$, namely we ask the scheme to be strongly stable, we can get $z\leqslant 8$, which is equivalent to ask $\tau\leqslant C\sqrt{h}$ where $C$ depends upon $\nu$. This is the desirable CFL condition.
\end{enumerate}

\subsection{High frequency regime}
The behavior of high frequency waves are usually considered in numerical research for its different phenomena which lower frequency waves don't have. For example, \cite{jefferis2015computing} developed numerical methods for high frequency solutions of general symmetric hyperbolic systems, and \cite{jin2008computation} focuses on the Liouville equation of geometric optics coupled with the Geometric Theory of Diffraction (GTD). Both of them use a WKB kind initial data, i.e. $u(x,0)=u_0(x)e^{iS_0(x)/\epsilon}$, where the scaled wavelength $\epsilon=L/K$ is small ($L$ is the considered length scale while $K$ is the typical wavenumber). In this paper, to capture the dynamics of such WKB kind initial data, a rescaling is done:
\begin{gather}
\left\{
\begin{split}
& u_t=\sigma \Lambda v,~~(x, t)\in \mathbb{T}_1\times [0,\infty)\\
& \epsilon v_t=-c\, u,~~(x, t)\in \mathbb{T}_1\times [0,\infty)
\end{split}
\right.
\end{gather}
with $\mathbb{T}_1=\mathbb{R}/\mathbb{Z}$ and the corresponding equation is:
\begin{equation}\label{rescaling}
\epsilon\partial_{tt}u=-\nu H\partial_xu.
\end{equation}
As what we did in the first section, multiplying by $u_t$ on both sides of \eqref{rescaling}, and intergrating over $\theta$ yields
\[
\frac{d}{dt}\int_{\mathbb{T}_1} \left(\epsilon|u_t|^2+ \frac{1}{2}\mu u\Lambda u \right) d\theta=0,
\]
which means the energy\[\int_{\mathbb{T}_1} \left(\epsilon|u_t|^2+ \frac{1}{2}\mu u\Lambda u \right) d\theta\]
is a conserved energy.
On the Fourier side, the second order system is
\begin{equation}\label{highfreq}
\partial_{t}\left(\begin{array}{c} \hat{u} \\
\hat{v} \end{array}\right) 
=
\left(  
\begin{array}{ccc}  
0 & \sigma |k| \\  
-\dfrac{c}{\epsilon} & 0 \\
\end{array}
\right)
\left(\begin{array}{c} \hat{u} \\
\hat{v} \end{array}\right):=A_1\left(\begin{array}{c} \hat{u} \\
\hat{v} \end{array}\right),
\end{equation}then the general framework of RK method for this system can be established in the same way, as long as parameter $c$ is replaced by $c/\epsilon$. Thus using the eigenvalues of $A_1$ instead of $A$ and rescaling $L$ in (b) of Theorem 1 are sufficient to derive the new stable condition, which can be summarized in the following theorem:
\begin{theorem} There are 2 cases in this theorem:
\begin{enumerate}
\item For any RK method, if a CFL condition holds, i.e. $\tau\leqslant \epsilon Ch$ for some positive real number $C$ irrelevant to $\epsilon$, the scheme for system \eqref{sys:Qsimilar}, \eqref{sys:Fourier} or \eqref{sys:QFourier} is stable. More specifically, let $\b{e}$ stand for vector of ones, then 
\begin{enumerate}
\item If $\mathrm{tr}(\b{G}^2)>\mathrm{tr}\left((\b{G}-\b{e}\b{w}^T)^2\right)$, then $\psi(\tau,\nu)<1$ when $\tau$ goes to 0, which means strong stability.
\item If $\mathrm{tr}(\b{G}^2)\leqslant\mathrm{tr}\left((\b{G}-\b{e}\b{w}^T)^2\right)$, then 
$\psi(\tau,\nu)\leqslant1+L\tau+O(\tau^2)$ when $\tau$ goes to 0, where $0\leqslant L\leqslant\mathrm{tr}\left((\b{G}-\b{e}\b{w}^T)^2-\b{G}^2\right)c\sigma C\pi/2\epsilon$, which means weak stability.
\end{enumerate} 
\item For any RK method whose absolute stable region contains a part of imaginary axis $[-C_1i,C_1i]$, then there exists $C_2>0$ such that when\begin{equation}\label{optimalstrategy}
\tau\leqslant C_2\sqrt{\epsilon h},
\end{equation} 
the scheme is strongly stable.
\end{enumerate}
\end{theorem} 
Formula \eqref{optimalstrategy} provides convinience for simulation of the behavior of high frequency wave, or system \eqref{highfreq}. A nortorious difficulty in numerical research of high frequency wave is the spatial aliasing error and its consequntial requirement of stability: if we want to capture the a high frequency wave function correctly, we need to resolve the spatial mesh, which also requires a smaller time step size. If the spatial step size $h\sim\epsilon$, we might ask the time step $\tau$ to be of higher order of $\epsilon$, such as $O(\epsilon^2)$. But in the context of system \eqref{highfreq}, if $h\sim\epsilon$ holds, using certain RK method, such as the classic RK4 we use in the later section, we can get that stability condition \eqref{optimalstrategy} merely asks $\tau\sim\epsilon$ holds. Thus to conduct a convincing numerical research of system \eqref{highfreq}, we don't have to require $\tau$ to be higher order of $\epsilon$, only $O(\epsilon)$ is already enought which won't lead to heavly load of computation. This explains the meaning of formula \eqref{optimalstrategy}: it verifies an optimal meshing strategy.

\section{Variable-coefficient system}\label{sec:variable}

In this section, we move onto System \eqref{eq:nonlocalsystem} with variable coefficients. We assume that all the coefficients are smooth on $\mathbb{T}$, $\sigma\ge \sigma_0>0$ and $c(\theta, t)\ge c_0>0$. Consider the energy functional
\begin{gather}\label{eq:energyfunctional}
\begin{split}
E&=\frac{1}{2}\left(\langle \Lambda v, v\rangle+\langle v, v\rangle+\left\langle \frac{c}{\sigma}u, u \right\rangle\right) \\
&=\frac{1}{2}\int_{\mathbb{T}} \left(v(\theta, t)\Lambda v(\theta, t)+v^2(\theta, t)+\frac{c(\theta, t)}{\sigma(\theta, t)} u^2(\theta, t)\right) d\theta.
\end{split}
\end{gather}
Taking the derivative, we find
\[
\begin{split}
\dot{E}&=\langle \Lambda v, v_t\rangle+\langle v, v_t\rangle
+\frac{1}{2}\left\langle \frac{d}{dt}\left(\frac{c}{\sigma}\right)u, u \right\rangle
+\left\langle \frac{c}{\sigma}u, u_t \right\rangle\\
&=\langle \Lambda v, g_2\rangle+\langle v, -cu+g_2\rangle
+\frac{1}{2}\left\langle \frac{d}{dt}\left(\frac{c}{\sigma}\right)u, u \right\rangle
+\left\langle \frac{c}{\sigma}u, g_1 \right\rangle
\end{split}
\]

By Parseval equality, we have $\langle \Lambda v, g_2\rangle
=\langle v, \Lambda g_2\rangle\le \|\Lambda g_2\|_2\|v\|_2$. Other terms can be easily controlled:
\[
\dot{E}\le C_1E+C_2\sqrt{E}.
\]
This implies that
\[
\frac{d}{dt}\sqrt{E}\le \frac{1}{2}(C_1\sqrt{E}+C_2).
\]
Gr\"onwall inequality implies that $\sqrt{E}$ is stable, which is consistent with the hyperbolicity.

\subsection{Convergence of the semi-discretization }
In this subsection, we discretize the spatial variables first and then show that the semi-discretization is convergent.
We approximate $u(\tau,t)$ and $v(\tau,t)$ respectively by $N$-vectors $U_h$ and $V_h$. Let $\sigma, a,b,c$ be restricted to the grid points.

Let us introduce a filter function $\xi\mapsto\rho(\xi)\in\mathbb{R}$ ($\xi\in(-\pi,\pi]$) Given a filter function $\rho: (-\pi, \pi]\to \mathbb{R}$, we denote $\check{\rho}_h$ the operator with symbol $\rho_h(k)=\rho(hk)$, so that
\begin{gather*}
g=\check{\rho}_hf ~~~\text{means}~~~\hat{g}_k=\rho(hk)\hat{f}_k, ~~~k\in[N]^*
\end{gather*}
Then, we have the filtered version:
\begin{gather}
\mathcal{D}_{\rho}=\check{\rho}_h\mathcal{D}f, ~~~\mathcal{L}_{\rho}=\check{\rho}_h\mathcal{L}f.
\end{gather}

We will assume the following conditions for the filter function:
\begin{condition}
\begin{itemize}
\item $\rho \ge 0$, even and $\rho \in C^2(-\pi, \pi]$ (Note that $\rho$ may not be $C^2$ on torus).
\item There exists $r\in \mathbb{N}_+$ such that
\begin{gather}\label{eq:filteracc}
\sup_{\xi\in (0, \pi)}|\xi|^{-r}|\rho(\xi)-1|<\infty.
\end{gather}
\end{itemize}
\end{condition}

Clearly, if we do not need a filter, we can simply set $\rho=1$. The centered difference on torus $(D_cu)_j=\frac{1}{2h}(u_{j+1}-u_{j-1})$ can be regarded as a filtered Fourier differentiation with filter $\rho(\xi)=\frac{\sin(\xi)}{\xi}$. Since $\rho$ is nonnegative, we can then define the natural discrete Sobolev norms associated with $\rho$ to be
\begin{gather}
\Vert f\Vert^2_{H^1_h}:=\Vert f\Vert^2_2+\Vert\mathcal{D}_{\rho} f\Vert^2_2, ~~~
\Vert f\Vert^2_{H^{1/2}_h}:=\sum_{k\in[N]^*}(1+|k|\rho(kh))|\hat{f}_k|^2.
\end{gather}

The following properties are straightforward from Lemma \ref{lmm:pars}:
\begin{lemma}\label{lmm:intbypartD}
Suppose $f,g$ are two $N$-vectors.We have the integration by parts following Parserval's equality:
\begin{gather*}
\langle f,\mathcal{D}_{\rho}g\rangle=-\langle \mathcal{D}_{\rho}f,g\rangle, ~~ \langle f,\mathcal{H}g\rangle=-\langle \mathcal{H}f,g\rangle,
~~\langle f,\mathcal{L}_{\rho}g\rangle=\langle \mathcal{L}_{\rho}f,g\rangle.
\end{gather*}
\end{lemma}
To see this, we check for example
\[
\langle f,\mathcal{D}_{\rho}g\rangle=
2\pi\sum_{k\in [N]^*} \hat{f}~ \overline{\rho(kh)ik\hat{g}_k}
=-2\pi \sum_{k\in [N]^*} \left(ik\rho(kh) \hat{f}_k\right)~ \overline{\hat{g}_k}
=-\langle \mathcal{D}_{\rho}f, g\rangle.
\]
Other equalities can be similarly checked and we omit the details.

With the filter, we discretize the system in space with the filtered pseudo Fourier spectral method, while keeping time continuous:
\begin{gather}\label{eq:UVh}
\left\{
\begin{split}
& \frac{dU_h}{dt} = \sigma\mathcal{L}_{\rho}V_h +g_1, \\
& \frac{dV_h}{dt} = -cU_h +g_2.
\end{split}
\right.
\end{gather}

We start with checking the consistency of the discretization. The following is straightforward by Fourier analysis and the aliasing formula, whose proof is omitted:
\begin{lemma}\label{lmm:filtererr}
Let $\varphi \in C^{\infty}(\mathbb{T})$ and $N\in\mathbb{N}$. Then,
the restriction $f=(f_j)=(\varphi(\theta_j))$ of $\varphi$ to the grid
points satisfies
\begin{gather}
\label{eq:aliasDf}
\begin{split}
&(\mathcal D_{\rho} f)_j-\varphi'(\theta_j)= R_1(\theta_j, h, r)h^r,\qquad
(\mathcal{L}_{\rho} f)_j - (\Lambda\varphi)(\theta_j) = R_2(\theta_j, h,r)h^r,\quad j\in[N],
\end{split}
\end{gather}
where $R_i: \mathbb{T}\to \mathbb{R}$ ($i=1,2$) are functions with $|\partial_{\theta}^\alpha R_i(\theta, h, r)|$ bounded uniformly in $\theta$ and $h$, for any $\alpha\in\mathbb{N}$.
\end{lemma}

As a corollary of Lemma \ref{lmm:filtererr}, we have the following consistency result, and the proofs are omitted:
\begin{lemma}\label{lmm:consistency}
Assume that the exact solution $(u,v)\in C^{\infty}(\mathbb{T}\times [0,T])$ and the filter satisfies \eqref{eq:filteracc}. Setting
$U_e=(u(\tau_j,t))$, $V_e=(v(\tau_j,t))$,
then we have
\begin{gather} \label{eq:UVconistency}
\left\{
\begin{split}
\frac{dU_e}{dt} = \sigma\mathcal{L}_{\rho} V_e+g_1+R_3(\theta_j, t;h) h^r,\\
\frac{dV_e}{dt} = -cU_e+g_2+R_4(\theta_j,t; h) h^r.
\end{split}
\right.
\end{gather}
where $R_i(\cdot,\cdot;h)$ ($i=3,4$) are two smooth functions on $\mathbb{T}\times [0, T]$ with $W^{\alpha, \ell^{\infty}}$ norms uniformly bounded in $h$ for any $\alpha\in \mathbb{N}$.
\end{lemma}

Now we show the convergence of the semi-discretized equations \eqref{eq:UVh}.
\begin{proposition}\label{pro:semidis_convergence}
Consider \eqref{eq:nonlocalsystem} with $\sigma\ge \sigma_0>0$ and $c\ge c_0>0$ and all the coefficients are smooth. For any $r\in \mathbb{N}$, the exact solution $(u,v)\in C^{\infty}(\mathbb{T}\times [0,T])$. Let $(U_e,V_e)$ be the restriction of the exact solution to grid and $(U_h,V_h)$ be the numerical solution given by the pseudo-spectral method \eqref{eq:UVh} with the same initial values. Then there exists a constant $M(T)>0$, such that $\forall t\in [0, T]$:
\begin{gather}
\begin{split}
\Vert U_h(t)-U_e(t)\Vert_2 \leq M(T)h^{r}, \\
\Vert V_h(t)-V_e(t)\Vert_{H^{1/2}_h} \leq M(T)h^{r}.
\end{split}
\end{gather}
\end{proposition}
\begin{proof}
Define the error vectors
\begin{gather}
e_u=U_h-U_e, ~~~ e_v=V_h-V_e.
\end{gather}
Taking the difference of equations \eqref{eq:UVh} and \eqref{eq:UVconistency}, we find the error functions satisfy the following equations
\begin{gather}\label{eq:erruv}
\left\{
\begin{split}
\frac{de_u}{dt}=\sigma\mathcal{L}_{\rho} e_v+R_3h^r, \\
\frac{de_v}{dt}=-ce_u+R_4h^r.
\end{split}
\right.
\end{gather}
Consider the energy functional for this ODE system
\begin{gather}
E=\frac{1}{2}\left(\langle\mathcal{L}_{\rho}e_v,e_v\rangle+\Vert e_v\Vert^2_2+\langle\frac{c}{\sigma}e_u,e_u\rangle \right).
\end{gather}
Note that $\langle\mathcal{L}_{\rho}e_v,e_v\rangle+\langle e_v,e_v\rangle=\Vert e_v\Vert^2_{H^{1/2}_h}$ and $\langle\frac{c}{\sigma}e_u,e_u\rangle$ is equivalent to $\|e_u\|_2^2$ (i.e. there exist $C_1>0, C_2>0$ such that $C_1\|e_u\|_2^2\le \langle\frac{c}{\sigma}e_u,e_u\rangle \le C_2\|e_u\|_2^2$).

By a similar computation with the continuous case, we can estimate the energy. In particular, we have
\begin{gather*}
\frac{dE}{dt}= \langle\mathcal{L}_{\rho} e_v,\frac{de_v}{dt}\rangle + \langle\frac{c}{\sigma}e_u,\frac{de_u}{dt}\rangle+\frac{1}{2}\langle\frac{d}{dt}(\frac{c}{\sigma})e_u,e_u\rangle+\langle e_v, \frac{d}{dt}e_v\rangle
\end{gather*}
According to Equation \eqref{eq:erruv},
\begin{equation}\label{eq:dEdt2}
\begin{split}
\langle\mathcal{L}_{\rho}e_v, \frac{de_v}{dt} \rangle + \langle\frac{c}{\sigma}e_u,\frac{de_u}{dt}\rangle
& = \langle R_4h^r, \mathcal{L}_{\rho}e_v\rangle +
\langle \frac{c}{\sigma}e_u, R_3h^r \rangle \\
& =\langle (\mathcal{L} R_4)h^r, e_v\rangle+\langle \frac{c}{\sigma}e_u, R_3h^r \rangle\\
&\le M_1\sqrt{E}h^r.
\end{split}
\end{equation}
In the first estimate, we have used the fact that $\mathcal{L}R_4$ is uniformly bounded by the smoothness of the error.
The last term $\frac{d\langle e_v,e_v\rangle}{dt}$ is straightforward:
\begin{equation*}
\begin{split}
\frac{d\langle e_v,e_v\rangle}{dt} & =2 \langle e_v,-ce_u + Rh^r\rangle \\
&\leq M_2(\Vert e_u\Vert^2_2
+ \Vert e_v\Vert^2_2 + \Vert e_v\Vert_2h^r)
\end{split}
\end{equation*}
We have
\begin{gather*}
\frac{dE}{dt}\leq M(E+\sqrt{E}h^{r}) \Rightarrow \frac{d}{dt}\sqrt{E}\le \frac{M}{2}(\sqrt{E}+h^r).
\end{gather*}
By Gr\"onwall inequality, we finally obtain
\begin{gather*}
\sqrt{E}\leq M(T)h^{r}, \qquad \forall ~~ 0\leq t\leq T.
\end{gather*}
which leads to our estimate for the error directly.
\end{proof}

\subsection{Time diescretization}\label{subsec:timevariable}

For the convenience of further discussion, we introduce the notion of smoothing operators, analogy to the big-$O$ notation, introduced in \cite{bhl96}:
\begin{definition}\label{def:AN}
Let $\mathcal{A}=\{A_N\}$ be a family of operators indexed by $N$. We define its action on $N$-vector $f$ as
$\mathcal{A}(f):=A_N(f)$.
We say $\mathcal{A}$ is $m$-th order smoothing, if exists $C>0$ independent of $N$ such that for any vector $f$ we have
\begin{gather*}
\Vert \mathcal{A}(\mathcal{D}^pf)\Vert_2\leq C\Vert f\Vert_2, ~~~
\Vert \mathcal{D}^p_{\rho}(\mathcal{A}(f))\Vert_2\leq C\Vert f\Vert_2, ~~~\forall 0\leq p\leq m.
\end{gather*}
If $\mathcal{A}$ is $m$-th order smoothing, we denote it as $\mathcal{A}_{-m}$.
\end{definition}
We note that $h\mathcal{D}_{\rho}=\mathcal{A}_0$ since $|kh|\le \pi$.
Recall a lemma from \cite{bhl96}
\begin{lemma}\label{lmm:commuhilbert}
Let $[\varphi, \mathcal{H}]\cdot=\varphi \mathcal{H}\cdot-\mathcal{H}(\varphi \cdot)$ be the commutator between $\varphi$ and discrete $\mathcal{H}$. Besides, the conditions for $\rho$, if we further have $\rho(\pi)=0$, then for $\varphi\in C^{\infty}$,
\begin{gather}
[\varphi,\mathcal{H}](\check{\rho}_h\omega)=\mathcal{A}_{-1}(\omega),~\forall \omega\in \mathscr{E}_N.
\end{gather}
If we instead have $\rho(\pi)=0$ and $\rho'(\pi)=0$, we have
\begin{gather}
[\varphi,\mathcal{H}](\check{\rho}_h\omega)=\mathcal{A}_{-2}(\omega),~\forall \omega\in \mathscr{E}_N.
\end{gather}
\end{lemma}

We denote define the operator
$A(t): \mathscr{E}_N^2\to \mathscr{E}_N^2$ as
\begin{gather}\label{eq:At}
A(t)(u, v)=\langle \sigma \mathcal{L}_{\rho}v, -c u\rangle
\end{gather}
so that \eqref{eq:UVh} can be rewritten as
\[
\frac{d}{dt}(U_h, V_h)=A(t)(U_h, V_h)+(g_1, g_2).
\]

We also define the operators $P(t): \mathscr{E}_N^2\to \mathscr{E}_N^2$ and $P^{-1}(t): \mathscr{E}\times\mathcal{Q}\to \mathscr{E}_N^2$ as
\begin{gather}
\begin{split}
& P(t)(u, v):=\left\langle u, \Lambda_{\rho}^{1/2}\Big(\sqrt{\frac{\sigma}{c}}v \Big) \right\rangle,\\
& P^{-1}(t)(u, v):=\left\langle u, \sqrt{\frac{c}{\sigma}}\Lambda_{\rho}^{-1/2}v \right\rangle.
\end{split}
\end{gather}
Here, the set $\mathcal{Q}$ is the following subspace of $\mathscr{E}_N$:
\[
\mathcal{Q}=\Big\{v\in \mathscr{E}_N: \|v\|_{\mathcal{Q}}:=\sum_{k\in [N]^*}\frac{1}{\rho(kh)|k|}|\hat{v}_k|^2<\infty \Big\}.
\]

We have the following claim
\begin{theorem}\label{thm:structureoftheoperator}
Suppose that the filter satisfies the conditions in ... and that $\rho(\pi)=0$. Then, we can decompose
the operator $A(t)$ (Equation \eqref{eq:At}) as
\begin{gather}
A(t)=\frac{1}{\sqrt{h}} P(t)^{-1}A_1(t) P(t)+P(t)^{-1}A_2(t)P(t),
\end{gather}
( recall $h=2\pi/N$), where the linear operators $A_1(t): \mathscr{E}_N^2\to \mathscr{E}_N^2$ and $A_2(t): \mathscr{E}_N\times\mathcal{Q}\to \mathscr{E}_N^2 $ satisfy
\begin{enumerate}[(i)]
\item The ranges of $A_i(t)$ are contained in $\mathscr{E}_N\times \mathcal{Q}$. $A_1(t)$ is anti-symmetric and there exist constants $N_0>0, C>0$ independent of $h$ such that
\[
\|A_1(t)(u,v)\|_2\le C(\|u\|_2+\|v\|_2), ~~\|A_2(t)(u, v)\|_2\le C\|v\|_{\mathcal{Q}}~\forall N\ge N_0.
\]
\item The eigenvalues of $P(t)^{-1}A_1(t) P(t)$ are purely imaginary, and bounded by a constant $C$ independent of $N$.
The eigenvalues of $P(t)^{-1}A_2(t)P(t)$ are bounded by a constant $C$ independent of $N$.
\end{enumerate}
\end{theorem}

\begin{proof}
We consider the operator $B(t)$ whose domain is $\mathscr{Q}$, defined by
\[
B(t):=P(t)A(t)P(t)^{-1}.
\]
Then, it is given by:
\[
B(t)(u, v)=\left\langle \sigma \Lambda_{\rho} \Big(\sqrt{\frac{c}{\sigma}}\Lambda_{\rho}^{-1/2}v\Big), -\Lambda_{\rho}^{1/2}(\sqrt{\sigma c} u) \right\rangle
\]

We then define $A_1(t)$ as
\[
\langle u, v\rangle \mapsto A_1(t)(u, v) :=\sqrt{h}\left\langle \sqrt{\sigma c} \Lambda_{\rho}^{1/2}v, -\Lambda_{\rho}^{1/2}(\sqrt{\sigma c} u) \right\rangle
\]
and $A_2(t):=B(t)-\frac{1}{\sqrt{h}}A_1(t)$ is given by
\[
A_2(t)(u, v)=\left\langle \sigma \left[\Lambda_{\rho}, \sqrt{\frac{c}{\sigma}}\right](\Lambda_{\rho}^{-1/2}v), 0 \right\rangle.
\]
We can directly verify that the ranges of $A_i$ are in $\mathscr{E}_N\times \mathcal{Q}$. That $A_1$ is bounded, antisymmetic is clear. We now focus on $A_2$. Note that
\[
\left[\Lambda_{\rho}, \sqrt{\frac{c}{\sigma}}\right](\Lambda_{\rho}^{-1/2}v)
=\mathcal{D}_{\rho}\left[\mathcal{H}, \sqrt{\frac{c}{\sigma}}\right](\Lambda_{\rho}^{-1/2}v)+\left[\mathcal{D}_{\rho}, \sqrt{\frac{c}{\sigma}}\right]\mathcal{H}\Lambda_{\rho}^{-1/2}v.
\]
Denote $w=\Lambda_{\rho}^{-1/2}v$ and it is clear that
\[
\|w\|_2\le C\|v\|_{\mathcal{Q}}.
\]
By Lemma \ref{lmm:commuhilbert}, the first term is
\[
\mathcal{D}_{\rho}\mathcal{A}_{-1}(w)
=\mathcal{A}_0(w).
\]
The second term, by the discrete product rule in \cite{bhl96} is also $\mathcal{A}_0(w)$. This then verifies $(i)$.

For (ii), we see that the action of $P(t)^{-1}A_i(t)P(t)$ ($i=1,2$) are well-defined for all $(u, v)\in\mathscr{E}_N^2$. Hence, they can be understood as a matrices. For $P(t)^{-1}A_1P(t)$, it is relatively easy to see the claim since $A_1$ is antisymmetric, bounded. We now focus on $P(t)^{-1}A_2P(t)$. Suppose that $(u,v)$ is a complex eigenvector in $\mathscr{E}_N^2$, so that
\[
P(t)^{-1}A_2P(t)(u, v)=\lambda (u,v).
\]
Denote $(u_1, v_1)=P(t)(u, v)$
\begin{gather*}
|\lambda|(\|u\|_2+\|v\|_2)=\|P(t)^{-1}A_2P(t)(u, v)\|_2=\|A_2P(t)(u, v)\|_2
\le C\|v_1\|_{\mathcal{Q}}\le C\|v\|_2.
\end{gather*}
This then shows (ii).
\end{proof}

By Theorem \ref{thm:structureoftheoperator}, we can find that the leading order structure is an anti-symmetric operator, whose eigenvalue is pure imaginary, and scales as $1/\sqrt{h}$. If we use ODE solvers whose stability region contains some part of the imaginary axis, like the explicit RK-$p$ method with $p\ge 3$, then we expect the stability condition is still
\[
\frac{\tau}{\sqrt{h}}\le C,
\]
for variable coefficient case.

\subsection{Comments on linear systems transport terms}\label{subsec:transport}

Consider the following linear systems:
\begin{gather}\label{eq:linearsystem2}
\begin{split}
& u_t=\sigma(\theta,t)\Lambda v+ b(\theta, t)\partial_{\theta}u+g_1,\\
& v_t=-c(\theta, t)u+b(\theta, t)\partial_{\theta}v+g_2
\end{split}
\end{gather}
where $\sigma, c, b$ are given coefficient functions. The transport terms affect the discretization in two aspects:

\begin{enumerate}[(i)]
\item
First of all, one may desire to use a filtered version of $\mathcal{L}_{\rho}$ and $\mathcal{D}_{\rho}$ with $\rho(\pi)=0$, $\rho'(\pi)=0$ to dampen high frequency so that the discretized energy is still stable.

To see this, let us consider the continuous version of the equations, and consider the same energy functional \eqref{eq:energyfunctional}.
The terms in $\dot{E}$ can be estimated similarly as before except for
$\langle \Lambda v, b\partial_{\theta}v\rangle$. To estimate this term, we find
\begin{gather*}
\int_{\mathbb{T}} (\Lambda v) b\partial_{\theta}v d\theta=-\frac{1}{2}\int_{\mathbb{T}} \partial_{\theta}v [H, b](\partial_{\theta} v)d\theta
=\frac{1}{2}\int_{\mathbb{T}} v \partial_{\theta}([H, b]\partial_{\theta}v)d\theta.
\end{gather*}
where $[H, b] =Hb-bH$ is the commutator. Note that $b$ is smooth. The commutator $[H, b]\partial_{\theta}v$ gives a convolution type integral between a smooth function with $\partial_{\theta}v$. It follows that
\[
\frac{1}{2}\int_{\mathbb{T}} v \partial_{\theta}([H, b]\partial_{\theta}v)d\theta
\le C\int_{\mathbb{T}} v^2 d\theta
\]
Hence, $\dot{E}\le C_1E+C_2\sqrt{E}$ still holds.

Unfortunately, for discretized Hilbert transform,
$[\mathcal{H}, b]$ is not in general smoothing. In fact, in \cite{bhl96}, the authors found that $[\mathcal{H}, b]$ may not even be $\mathcal{A}_{-1}$. By Lemma \ref{lmm:commuhilbert}, one needs a filter $\rho$ so that the commutator $[\mathcal{H}, b]\check{\rho}=\mathcal{A}_{-2}$ has the smoothing effect so that \[
\dot{E}\le C_1 E+C_2\sqrt{E}
\]
still holds.

\item On the other side, the transport terms require the CFL condition to be of the form
\[
\frac{\tau}{h}\le C,
\]
which is more restrictive compared with \eqref{eq:newcfl}, which could be resolved using semi-Lagrangian method.
\end{enumerate}

\section{Relation to water wave simulation}\label{sec:towaterwave}

In Section \ref{sec:intro}, we mentioned that the nonlocal hyperbolic systems are closely related to waterwave equations and their simulation. In this section, we explore this in more detail and see how our study of the nonlocal system implies the stability conditions for the simulation of waterwave problems.

Consider a two dimensional fluid (water) with infinite depth. Suppose the waves are periodic so that the surface of the fluid can be described by $z: \mathbb{R}\times [0,\infty) \rightarrow \mathbb{C}$:
\begin{gather}
z(\alpha, t)=x(\alpha, t)+i y(\alpha, t),
\end{gather}
where $\alpha\in \mathbb{R}$ parametrizes the undisturbed surface so that $\alpha$ is a material coordinate. Also, $s(\alpha, t):=z(\alpha, t)-\alpha$ is a periodic function in $\alpha$. Without loss of generality, we can assume that the period is $2\pi$ (otherwise, one can just do rescaling).

The fluid is inviscid and irrotational so that there exists a velocity potential $\Phi(x,y,t)$ such that the velocity is given by $\nabla\Phi$. Let
\[
\phi: \mathbb{R}\times [0,\infty) \rightarrow \mathbb{R}, ~~(\alpha, t)\mapsto \phi(\alpha, t):=\Phi(x(\alpha, t), y(\alpha, t), t)
\]
be the evaluation of the velocity potential at the surface so that $\phi(\alpha, t)$ is a periodic function in $\alpha$ with period $2\pi$. By the derivation in
\cite{bmo1982, beale93, bhl96}, $z$ and $\phi$ satisfy the following system of equations (Equations (1)-(3) in \cite{bhl96}):
\begin{gather}\label{eq:waterwave}
\left\{
\begin{split}
&\bar{z}_t=\frac{1}{4\pi i}\int_{-\pi}^{\pi} \gamma(\alpha') \cot(\frac{z(\alpha)-z(\alpha')}{2})d\alpha' +\frac{\gamma(\alpha)}{2 z_{\alpha}(\alpha)}=: w(\alpha, t),\\
& \phi_t=\frac{1}{2}|w|^2-gy,\\
& \phi_{\alpha}=\frac{\gamma}{2}+Re\left[ \frac{z_{\alpha}}{4\pi i}\int_{-\pi}^{\pi}\gamma(\alpha')
\cot(\frac{z(\alpha)-z(\alpha')}{2}) d\alpha' \right],
\end{split}
\right.
\end{gather}
where $\bar{z}$ means the complex conjugate, and $\gamma(\alpha, t)$ is some unknown quantity to be determined by the third equation.

In \cite{beale93}, the authors showed that the linearization of \eqref{eq:waterwave} leads to \eqref{eq:bealeerr}. Indeed, this happens for the numerical schemes as well. In \cite{bhl96}, the authors then proposed a filtered pseudo-spectral differentiation method to discretize the spatial variables. Besides the conditions in
\tcr{to fill in}, assume the filter also satisfies (i) $\rho(\pi)=0$ and $\rho'(\pi)=0$; (ii) there exists $r\ge 4$ such that $|\rho(\xi)-1|\le C|\xi|^r$ for $\xi\in (-\pi, \pi]$.

Let $j\in [N]$ and $(z_j, \phi_j, \gamma_j)$ be the numerical solutions at the grid points. Then, the discretization to \eqref{eq:waterwave} is given by (Eq. (7)-(9) in \cite{bhl96})
\begin{gather}\label{eq:waterwavesemi}
\left\{
\begin{split}
& \frac{d}{dt}\bar{z}_j=\frac{1}{4\pi i}\sum_{p=-N/2+1, ~p-j\mathrm{~odd}}^{N/2}\gamma_p \cot(\frac{\check\rho z_j-\check\rho z_p}{2}) 2h+\frac{\gamma_j}{2 (1+\mathcal{D}_{\rho}(z_j-\alpha_j))}=: w_j ,\\
& \frac{d}{dt}\phi_j=\frac{1}{2}|w_j|^2-g y_j,\\
& \mathcal{D}_{\rho}\phi_j=\frac{\gamma_j}{2}+Re\left[\frac{1+\mathcal{D}_{\rho} (z_j-\alpha_j)}{4\pi i}
\sum_{p=-N/2+1, ~p-j\mathrm{~odd}}^{N/2}\gamma_p \cot(\frac{\check\rho z_j-\check\rho z_p}{2}) 2h
\right]
\end{split}
\right.
\end{gather}
In \cite{bhl96}, the authors then introduced the following numerical errors
\begin{gather}
\begin{split}
&\eta_j(t)= Im\left[ (z_j(t)-z(\alpha_j, t))\frac{\overline{z}_{\alpha}(\alpha_j, t)}{|z_{\alpha}(\alpha_j, t)|}\right],\\
&\delta_j(t)= Re\left[ (z_j(t)-z(\alpha_j, t))\frac{\overline{z}_{\alpha}(\alpha_j, t)}{|z_{\alpha}(\alpha_j, t)|}\right]+(\mathcal{H}\eta)_j,\\
& \zeta_j=(\phi_j(t)-\phi(\alpha_j, t))-Re( w_j (z_j(t)-z(\alpha_j, t))
\end{split}
\end{gather}
By assuming the the strong Taylor sign condition (Eq. (88) in \cite{bhl96})
\[
c(\alpha, t):=-\partial_np\ge c_0>0
\]
and that the true solutions are smooth with
\[
\sigma(\alpha, t):=\frac{1}{|z_{\alpha}|}\ge \sigma_0>0,
\]
the authors found that these variables for errors satisfy the following semi-linear non-local hyperbolic system (Eq. (89)-(91)):
\begin{gather}\label{eq:bealfull}
\begin{split}
&\partial_t \eta_j=\sigma(\alpha_j, t) (\Lambda \zeta)_j+\mathcal{A}_0(\eta, \delta)+\mathcal{A}_0(\zeta)+R_5(h)h^r,\\
&\partial_t \zeta_j=-c(\alpha_j, t) \eta_j+\frac{1}{2}|w_j(t)-w(\alpha_j, t)|^2,\\
&\partial_t \delta_j=\mathcal{A}_0(\eta, \delta, \zeta)+R_6(h)h^r.
\end{split}
\end{gather}
See Definition \ref{def:AN} for $\mathcal{A}_0$.
The leading order behavior of the semi-linear system \eqref{eq:bealfull} is \eqref{eq:bealeerr}, the nonlocal hyperbolic system.
Making use of this hyperbolic structure, the authors showed that the semi-discrete system \eqref{eq:waterwavesemi} converges to the original waterwave problem \eqref{eq:waterwave}.
However, there was no discussion about time discretization.

Since the errors satisfy to the leading order the nonlocal hyperbolic system \eqref{eq:bealeerr} or \eqref{eq:nonlocalsystem}, by the discussion in \eqref{subsec:timevariable}, we expect that the stability conditions for time discretization would be similar to those for \eqref{eq:nonlocalsystem}. If use the scheme \eqref{eq:waterwavesemi} and RK4 for time discretization, we could have a relaxed constraint
\[
\frac{\tau}{\sqrt{h}}\le C,
\]
for stability. We will examine this numerically in Section \ref{sec:num}.

\begin{remark}
According to \eqref{eq:wu97}, if we discretize the water wave problem based on the conformal mapping method in \cite{wu97} instead of discretizing using the Lagrangian formulation as in \cite{bhl96}, intrinsically, we will have a transport term for the numerical error. By the discussion in Section \ref{subsec:transport}, we will need a more restrictive requirement
\[
\tau/h\le C,
\]
for the stability.
\end{remark}

\section{Numerical examples}\label{sec:num}

In this section, we present some numerical examples to verify our conclusion and carry out some meaningful numerical experiments. In Sections \ref{subsec:stabnonlocal} and \ref{subsec:waterwave}, we verify that when we apply RK4 temporal discretization the stability condition agrees with \eqref{eq:newcfl} for both the nonlocal hyperbolic system and water wave problem. Convergence of the discretization of nonlocal hyperbolic system is demonstrated in Section \ref{subsec:conv} and the exploration of high frequency regime of the nonlocal system is performed in Section \ref{subsec:highfreq}. Besides, a turn-over wave example in \cite{bhl96} is recovered to verify the correctness of our code in Section \ref{subsec:waterwave}. Periodic boundary conditions are chosen in these simulations and we always take spectral method or filtered spectral method for spatial discretization.

\subsection{Stability condition for nonlocal hyperbolic system}\label{subsec:stabnonlocal}

\begin{figure}
\begin{center}
\includegraphics[width=0.8\textwidth]{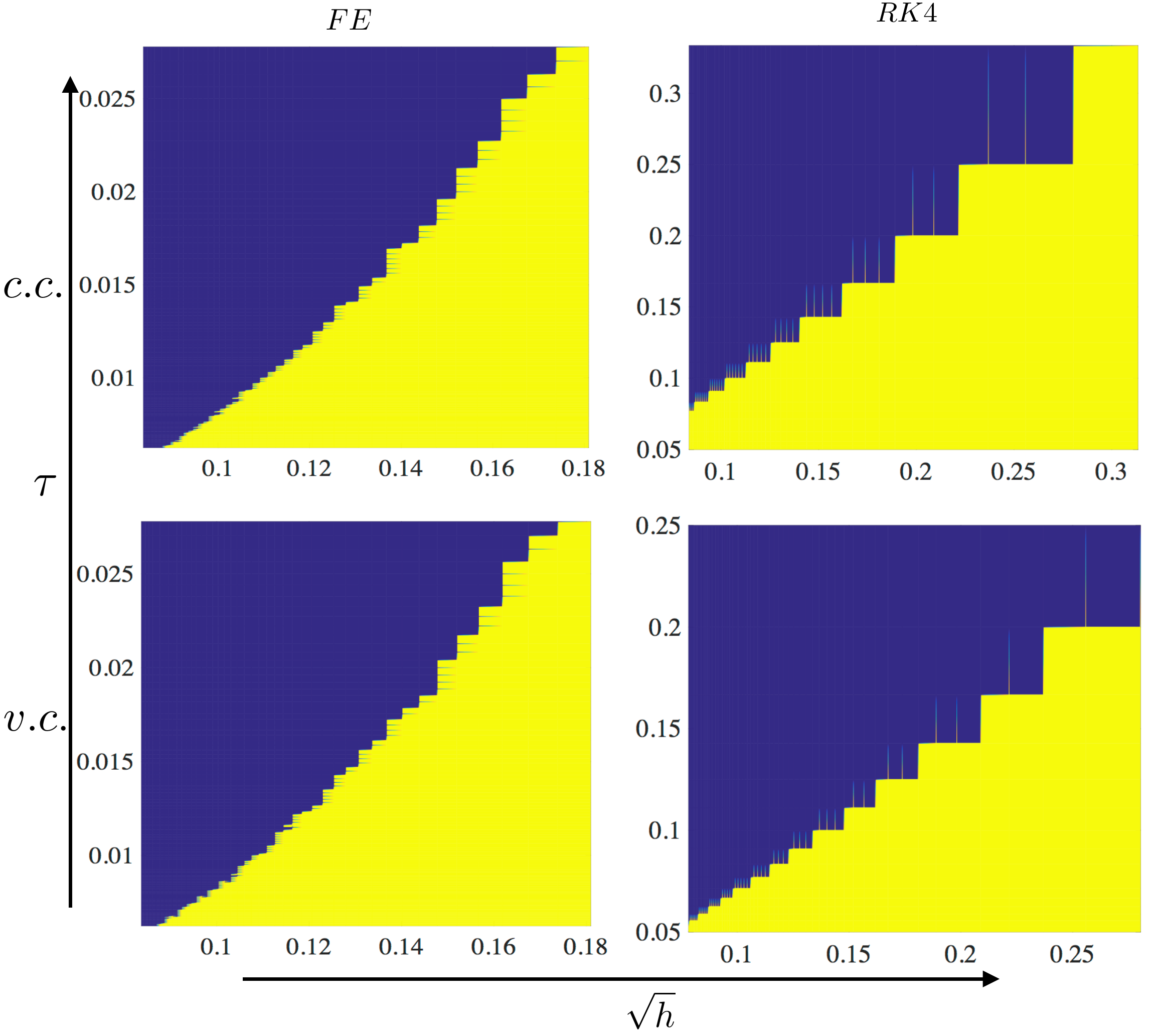}\\
\includegraphics[width=0.78\textwidth]{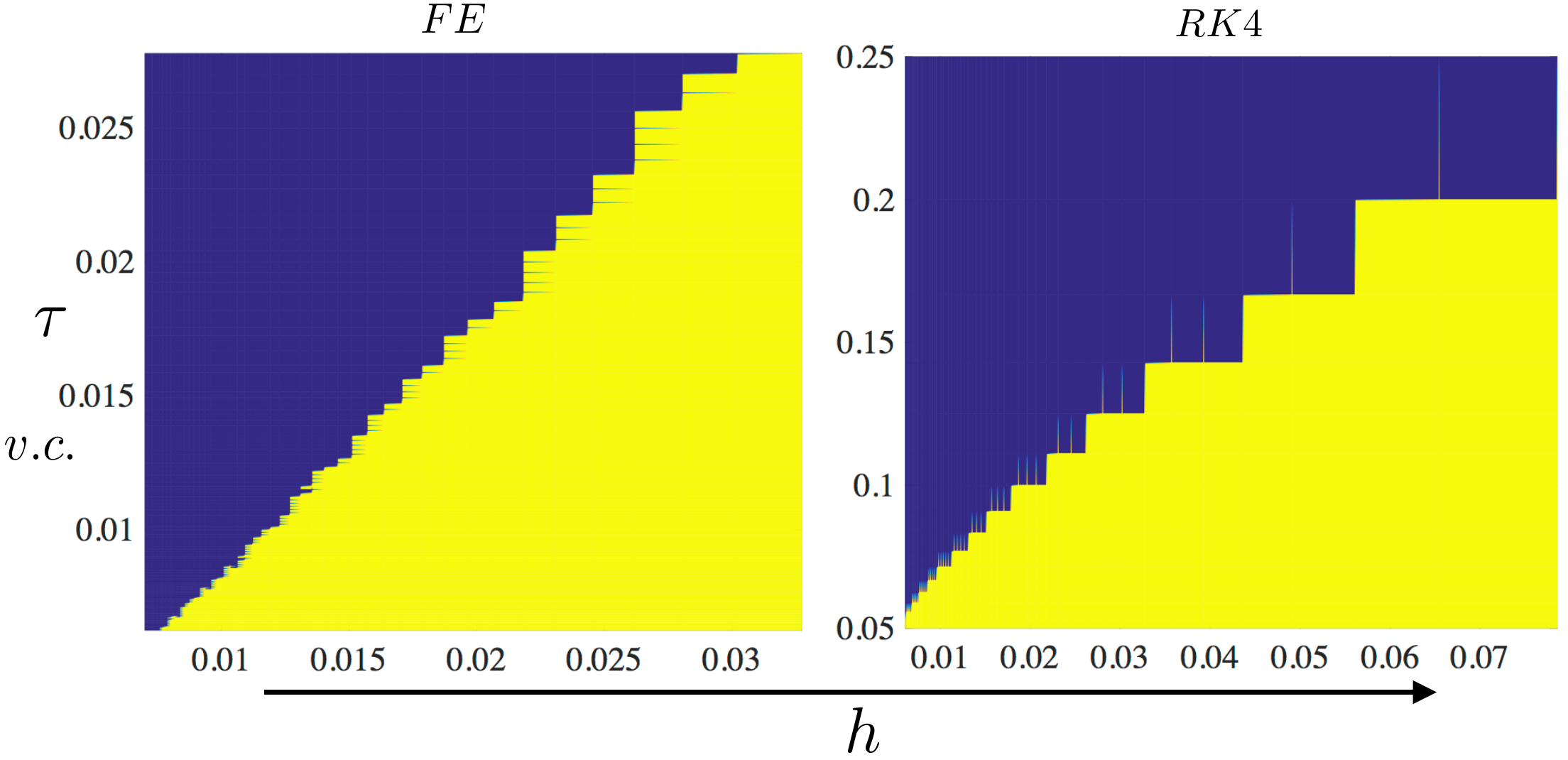}
\end{center}
\caption{Stability condition for nonlocal hyperbolic system. The first row shows the results for the constant coefficient (c.c.) case, while the second row shows the results for the variable coefficient case (v.c.). The two columns are for forward Euler and RK4 respectively. Bottom: Re-plot of second row into $h$-$\tau$ plane.}
\label{fig:gridfornonlocalhy}
\end{figure}

In this example, we test stability condition for the nonlocal hyperbolic system \eqref{eq:nonlocalsystem} with $g_1 = 0, g_2 = 0$.
We consider both the constant-coefficient case with
\[
c=3,~~\sigma=1
\]
and the variable-coefficient case with
\[
c(\theta, t)=\exp(\cos(\theta+t)),~~\sigma=2+\sin(\theta+t).
\]
We then perform the simulations for various stepsizes using Fourier spectral method in space and forward Euler (FE) and Runga-Kutta 4 (RK4) for temporal discretizaiton. The solutions are computed up to $T=10$. Results are presented in Figure \ref{fig:gridfornonlocalhy}, the blue part indicates the unstable region while the yellow part represents the stable region.

In the top half of Figure \ref{fig:gridfornonlocalhy}, we plot in the $\sqrt{h}$-$\tau$ plane. We find that the border for RK4 is like lines while the border curves for FE is some convex curve. This means that the stability condition for RK4 is really \eqref{eq:newcfl}. To check the condition for FE, we re-plot the variable-coefficient case in $h$-$\tau$ plane as shown in the bottom of Figure \ref{fig:gridfornonlocalhy}. The new plots show that the stability condition for FE is $\tau\le Ch$.

To understand these results, we recall that the stability region for FE only intersects the imaginary axis at $z=0$, while RK4 contains some part of imaginary axis. This verifies the analysis in Sections \ref{sec:constantcoe} and \ref{sec:variable}.

\subsection{Convergence study}\label{subsec:conv}

In this subsection, we verify the convergence numerically for the nonlocal hyperbolic system \eqref{eq:nonlocalsystem} with $g_1 = 0, g_2 = 0$, the
constant coefficients are given by
\[
c=3,~\sigma=1.
\]
The initial conditions are given by
\[
u_0(\theta)=e^{\sin(\theta)}+\cos(\theta), ~~v_0(\theta)=\cos^2(\theta).
\]
Again Fourier spectral method is used for spatial discretization and forward Euler (FE), backward Euler (BE), Crank-Nicolson (CN) and Runga-Kutta 4 (RK4) are used for temporal discretizaiton.

All results are computed to time $T=2$. The reference solution (or `accurate solution') is computed using Runga-Kutta 4 with $h=2\pi/2^7$ and $\tau=10^{-5}$. The error plots are shown in Figure \ref{fig:convergencefornonlocalhy}. By Figure \ref{fig:convergencefornonlocalhy} (a), it is clear that spectral accuracy is observed in spatial discretization. When $h\approx 0.2$, the errors have already been dominated by the temporal error. By Figure \ref{fig:convergencefornonlocalhy} (b) indicates that the temporal errors are of the order as expected. Our discretization schemes indeed converge.

\begin{figure}[H]
\begin{center}
\includegraphics[width=0.8\textwidth]{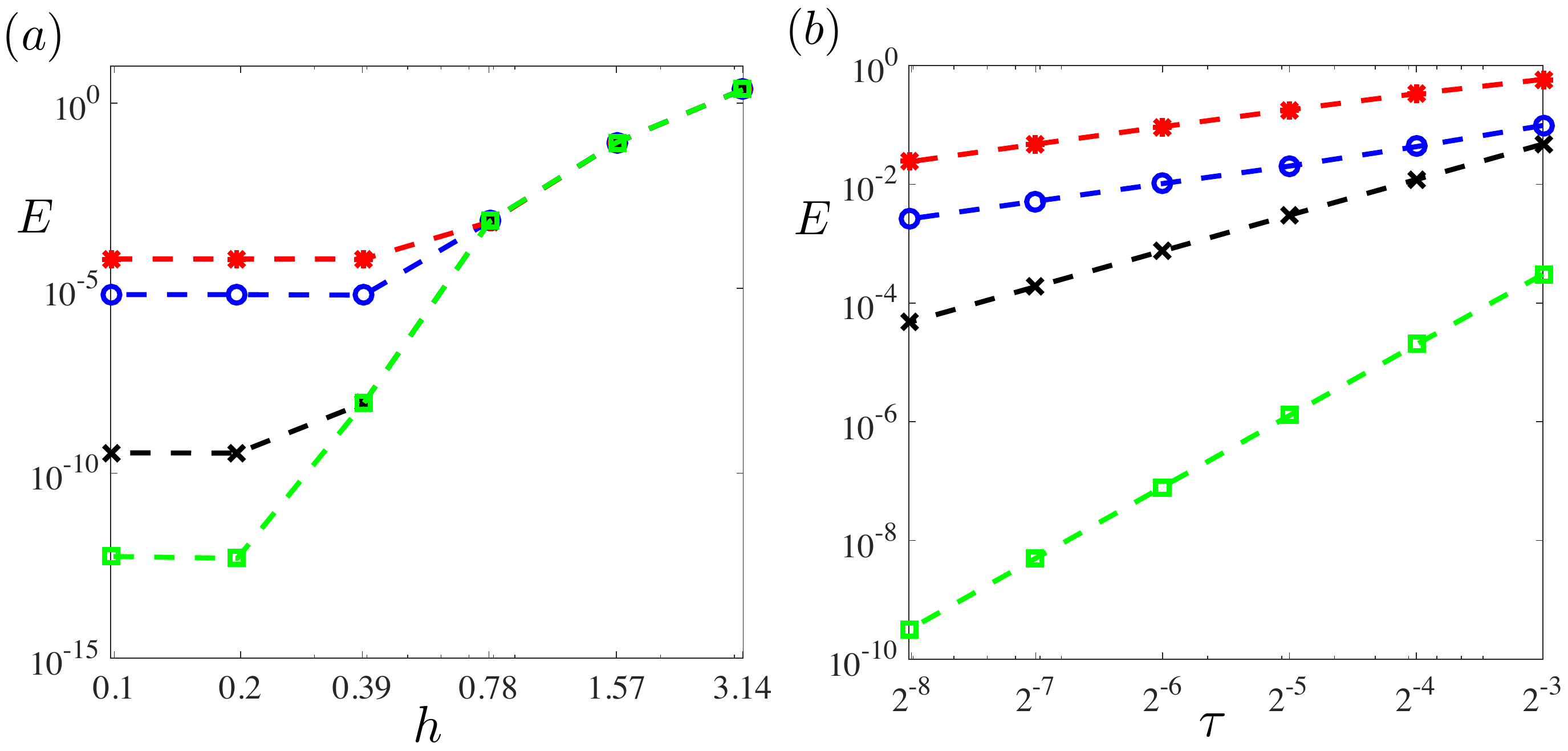}
\end{center}
\caption{Convergence study for forward Euler (blue circles), backward Euler (red stars), Crank-Nicolson (black crosses) and Runga-Kutta 4 (green squares). (a). Spectral convergence in spatial with $\tau=10^{-5}$ (b). Time convergence, with $h=2\pi/2^7$.}
\label{fig:convergencefornonlocalhy}
\end{figure}

\subsection{The system in high frequency regime}\label{subsec:highfreq}

To investigate the system in high frequency and check whether there is caustic phenomenon, we typically use a WKB kind initial value and require it to have a corresopond frequency. Therefore, we consider \eqref{rescaling} with a selected initial value:
\begin{equation}
\left\{
\begin{split}
& \epsilon\partial_{tt}u=-\mu H\partial_xu,(x,t)\in \mathbb{T}_1\times\mathbb{R}^{+}\\
& u(x,0)=e^{-100{(x-0.5)}^2}e^{i\mathrm{log}(20\mathrm{cosh}(5x-2.5))/\epsilon},\\
& u_t(x,0)=e^{-100{(x-0.5)}^2}e^{i\mathrm{log}(20\mathrm{cosh}(5x-2.5))/\epsilon}.
\end{split}
\right.
\end{equation}
The initial value is made up by a guassian function and a high frequency term, the former is used to control the support and the latter is a WKB type function.
In the numerical experiment, we choose $\mu=1$ and run for different $\epsilon$, namely $\epsilon=2^{-i},i=4,5,...,12$. We plot the snapshot of amplitude at $t=0.0625$ for $\epsilon=2^{-4},2^{-6},2^{-8},2^{-10}$ in Figure \ref{fig:hfreq}. We can find that the amplitude gets larger when $\epsilon$ gets smaller. Meanwhile, notice that no matter how small $\epsilon$ is, the amplitude of $u(x,0)$ is not larger than 1, thus the growing trend of amplitude provides evidence for caustic phenomenon. To conduct a more careful observation, we check the maximum amplitude before time $t=0.0625$ in the whole domain [0,1], and we plot the following log-log figure in Figure \ref{fig:loglog}.
\begin{figure}[H]
\begin{center}
\includegraphics[width=0.8\textwidth]{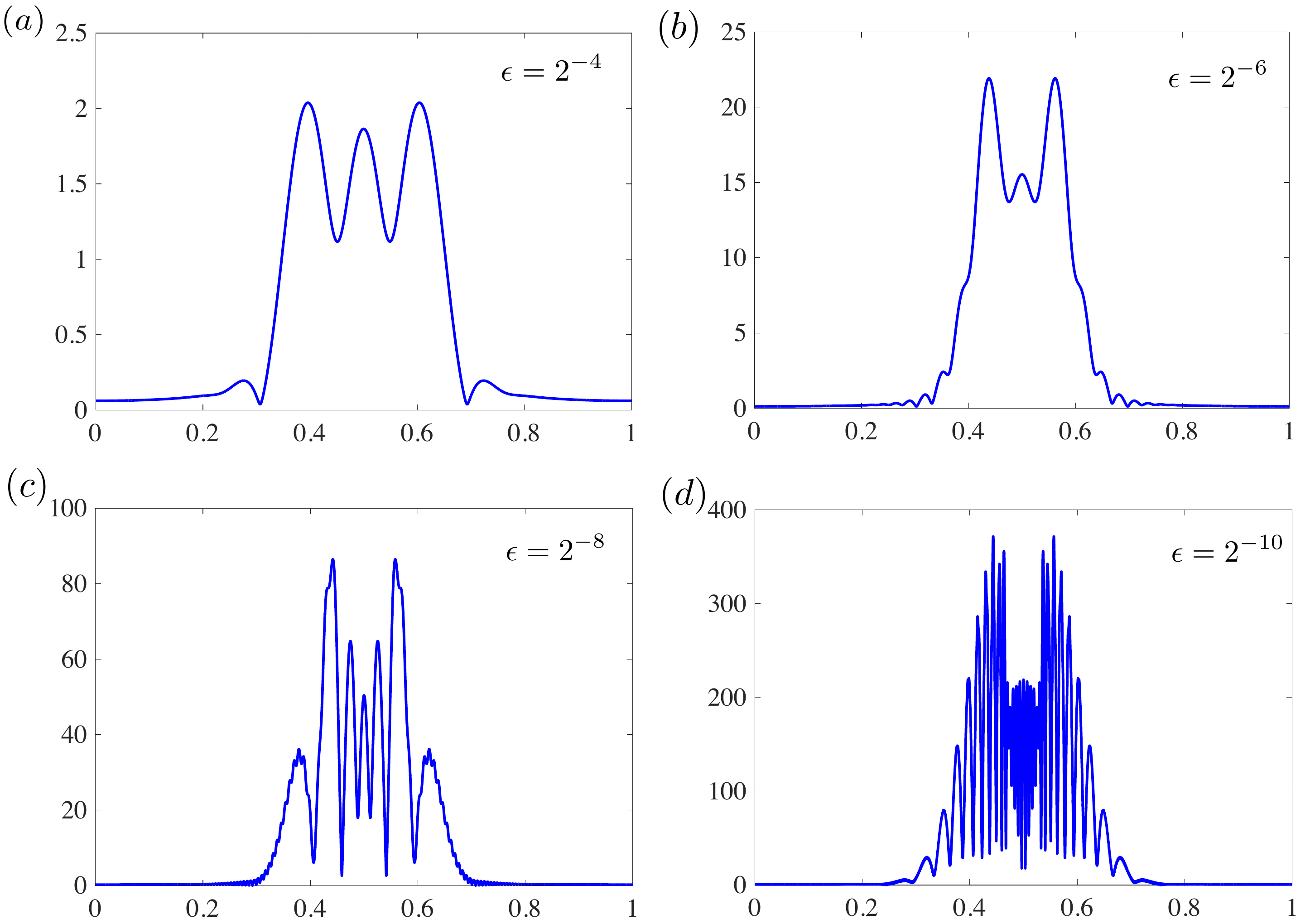}
\end{center}
\caption{Snapshot for amplitude $|u|$ versus $x$ at $t=0.0625$ for different $\epsilon$s. (a)$\sim$(d) for $\epsilon=2^{-4},2^{-6},2^{-8},2^{-10}$ respectively.}
\label{fig:hfreq}
\end{figure}
\begin{figure}[H]
\begin{center}
\includegraphics[width=0.5\textwidth]{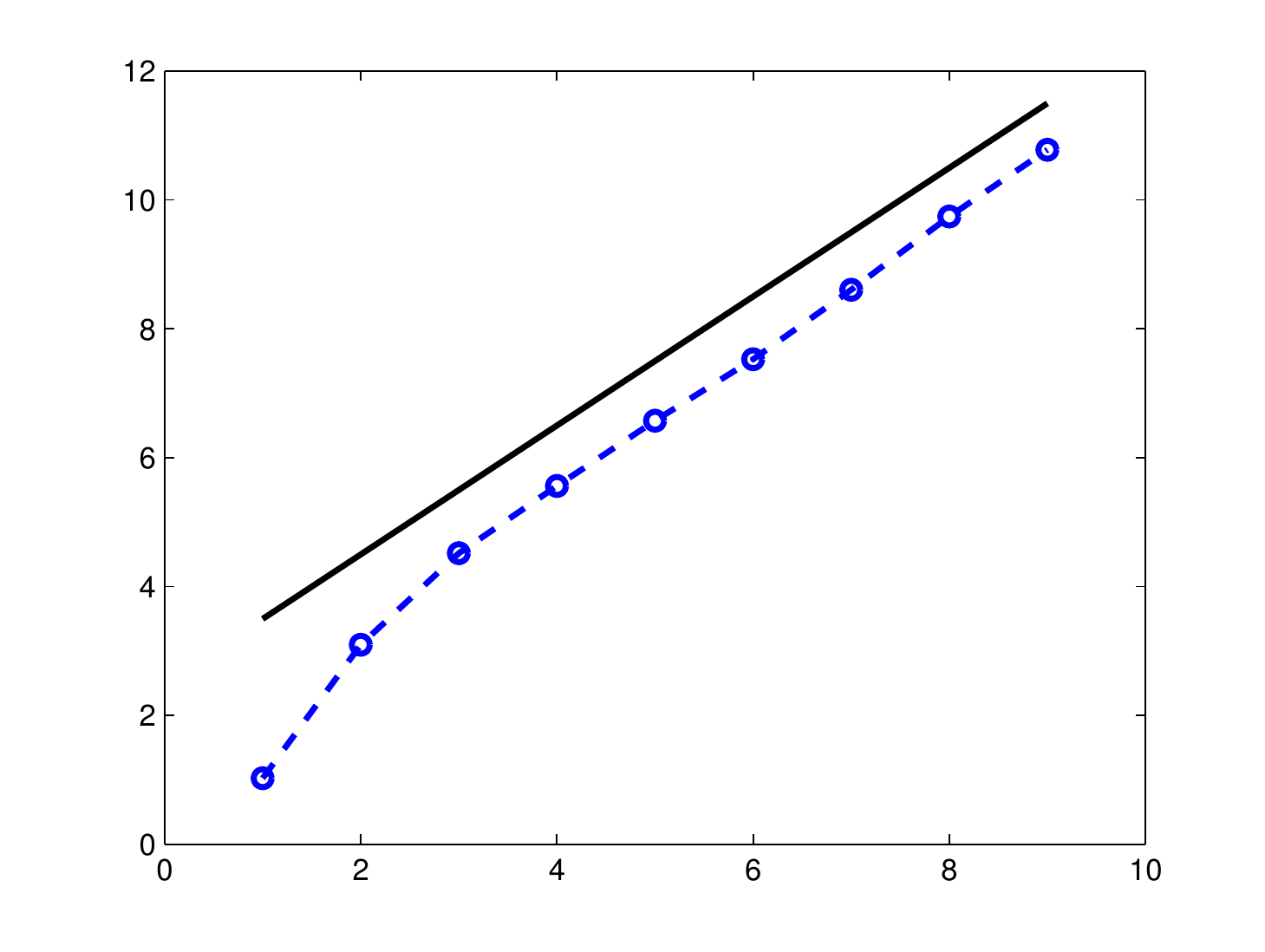}
\end{center}
\caption{Blue circles represent the plot of $y=\max\limits_{t\leqslant 0.0625,x\in[0,1]}\mathrm{log}_2{|u(x,t)|}$ versus $z=-\mathrm{log}_2{\epsilon}-3$, where $\epsilon=2^{-i},i=4,5,...,12$. Meanwhile, a reference line with a slope of 1 is also drawn to show the quantitative relation between $y$ (or amplitude) and $z$ (or $\epsilon$).}
\label{fig:loglog}
\end{figure}
From Figure \ref{fig:loglog}, we observe that when $\epsilon$ is sufficiently small, the curve is almost a line with a slope of 1, which indicates that the maximum amplitude is approximately proportion to $1/\epsilon$. This supports the existence of caustic phenomenon. 

\subsection{Stability condition for water wave simulation}\label{subsec:waterwave}

In this example, we perform the water wave simulation (Equation \eqref{eq:waterwave} with $\alpha\in \mathbb{T}$). The spatial discretization is implemented using filtered Fourier spectral method.
The filter we use in this section is given by
\begin{gather}\label{eq:filter1}
\rho(\xi)=\exp(-10 (\left|\xi\right|/\pi)^{25}),~\xi\in (-\pi, \pi],
\end{gather}
for which the condition $\rho(\pi)=0$ and $\rho'(\pi)=0$ are numerically satisfied.

To verify that our code runs correctly, we first test the same example in \cite{bhl96} to recover the turn over phenomenon. The initial data are given by:
\begin{gather}
\begin{split}
& x(\alpha,0)=\alpha,\\
& y(\alpha, 0)=0.6\cos(\alpha),\\
& \gamma(\alpha, 0)=1+0.6\sin(\alpha).
\end{split}
\end{gather}
The numerical solution is calculated by using $h = 1/512, \tau = 1/4000$, we use RK4 for time discretization here. The snapshots of the waves at different times are shown in Figure \ref{fig:turnover}. We have recovered exactly the same numerical results in \cite{bhl96}.
\begin{figure}[H]
\begin{center}
\includegraphics[width=0.8\textwidth]{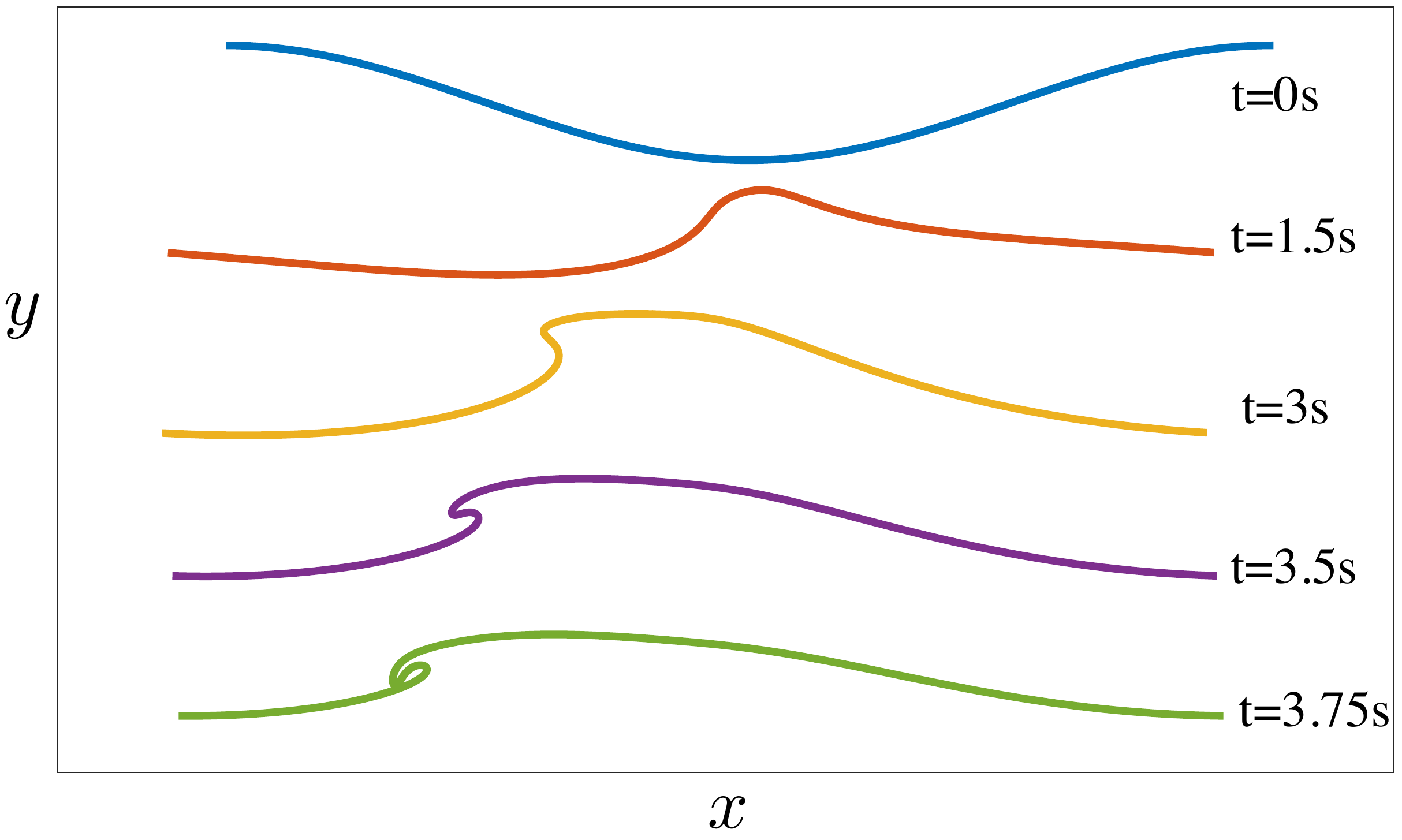}
\end{center}
\caption{Turn over of water waves. As time increases, the water wave turns over gradually. When the time is close to 3.75, the wave is going to break. }
\label{fig:turnover}
\end{figure}
Now, we study numerically the stability condition for the discretization for the water wave problem. The initial data used are given as
\begin{gather}
\begin{split}
& x(\alpha,0)=\alpha,\\
& y(\alpha, 0)=0.3\cos(\alpha),\\
& \gamma(\alpha, 0)=1+0.3\sin(\alpha).
\end{split}
\end{gather}
The spatial discretization is performed using filtered Fourier spectral method with the same filter (Equation \eqref{eq:filter1}) and we test FE and RK4 as temporal discretiztions. The simulations are performed up to time $T=4$.

The results are presented in Figure \ref{fig:wwsqh}. Same as in Section \ref{subsec:stabnonlocal}, the blue part indicates the unstable region while the yellow part represents the stable region. From this figure, we can tell that that stability condition for RK4 is like $\tau \lesssim \sqrt{h}$, while stability condition for FE is like $\tau \lesssim h$, in accord with our conclusion drawn in Section \ref{sec:towaterwave}.

\begin{figure}[H]
\begin{center}
\includegraphics[width=0.7\textwidth]{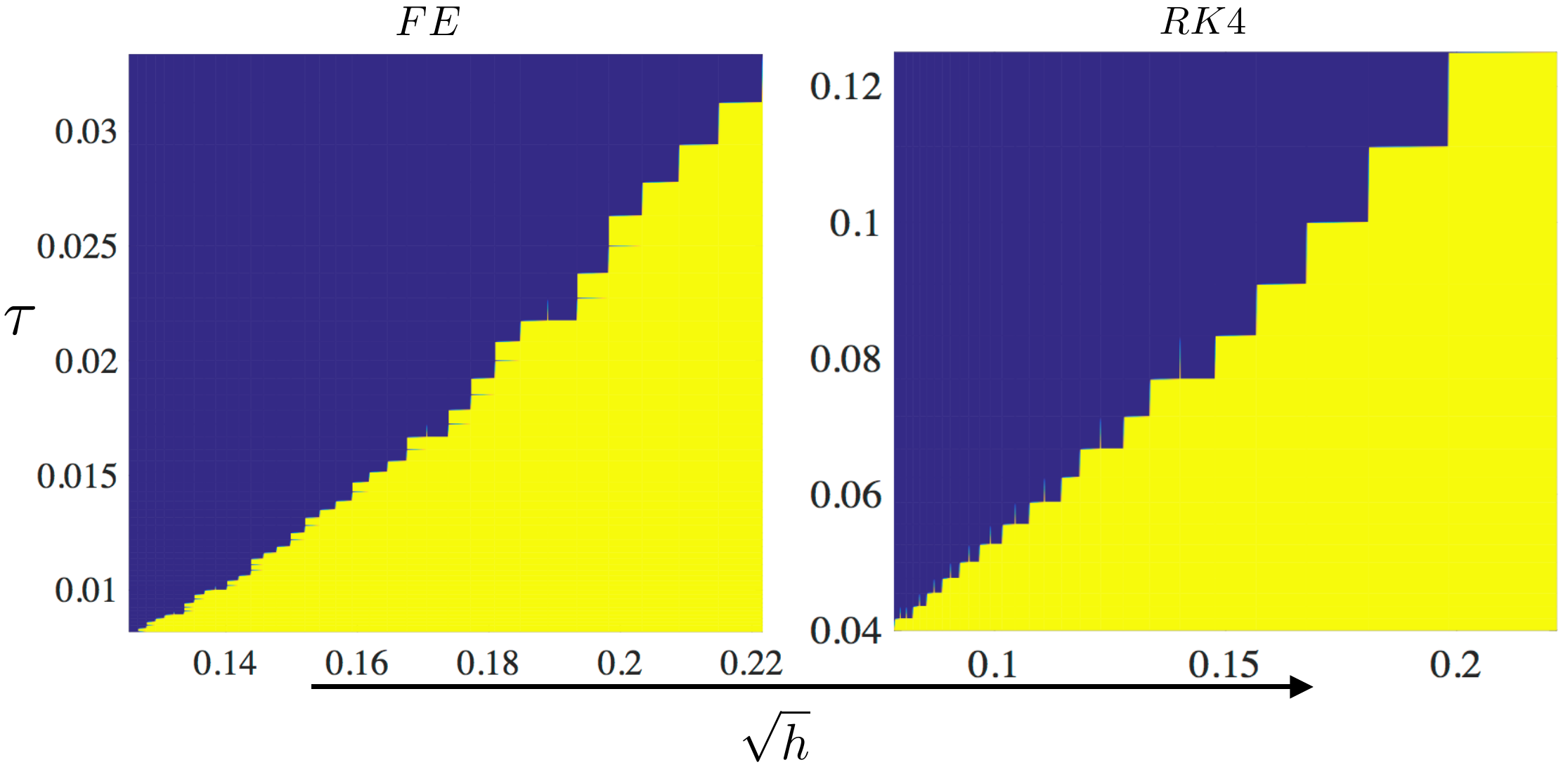}\\
\includegraphics[width=0.7\textwidth]{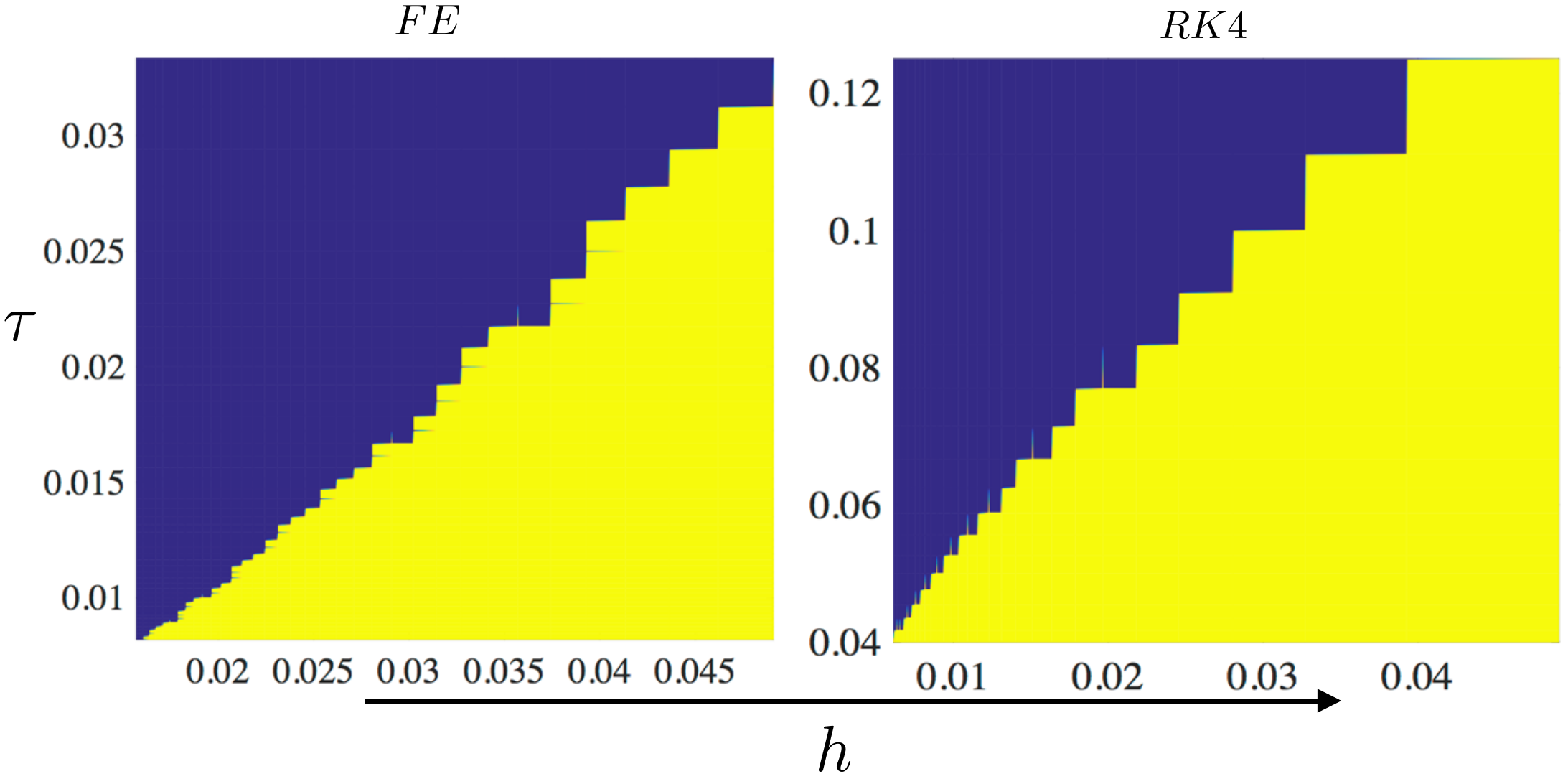}
\end{center}
\caption{stability for waterwave. From this plot, we see that stability conditiosn for RK4 is like $\tau \lesssim \sqrt{h}$, and stability condition for FE is more restrictive. Bottom: we see that stability condition for FE is like $\tau \lesssim h$, while stability condition for RK4 is better.}
\label{fig:wwsqh}
\end{figure}
\section*{Acknowledgements}
J.-G.Liu is supported in part by National Science Foundation (NSF) under award DMS-1514826. Z. Liu is supported by the Elite Undergraduate Training Program of the School of Mathematical Sciences at Peking University. Z. Zhen is partially supported by RNMS11-07444(KI-Net) and a start-up grant from Peking University.
\bibliographystyle{unsrt}
\bibliography{NumAnaPde}
\end{document}